\documentclass[reqno, a4paper]{amsart}
\usepackage{graphicx} 

\usepackage[english]{babel} 
\usepackage[utf8]{inputenc} 
\usepackage{amsmath} 
\usepackage{amssymb} 
\usepackage{amsthm} 
\usepackage{mathtools}
\usepackage{dsfont} 
\usepackage{graphicx} 
\usepackage{appendix} 
\usepackage{bbm}
\usepackage{amssymb}  
\usepackage{hyperref}
\usepackage{subfig}

\theoremstyle{plain}
\newtheorem{thm}{Theorem}[section]
\newtheorem{cor}[thm]{Corollary}
\newtheorem{lem}[thm]{Lemma}
\newtheorem{prop}[thm]{Proposition}

\theoremstyle{definition}
\newtheorem{ex}{Example}[section]
\newtheorem{defn}{Definition}[section]

\theoremstyle{remark}
\newtheorem{rmk}{Remark}
\newcommand{\numberset}{\mathbb}
\newcommand{\N}{\numberset{N}}

\newcommand{\R}{\numberset{R}}
\newcommand{\Q}{\numberset{Q}}
\newcommand{\E}{\numberset{E}}

\usepackage{listings}
\usepackage[dvipsnames]{xcolor}

\definecolor{codegreen}{rgb}{0,0.6,0}
\definecolor{codegray}{rgb}{0.5,0.5,0.5}
\definecolor{codepurple}{rgb}{0.58,0,0.82}
\definecolor{backcolour}{rgb}{0.95,0.95,0.92}

\lstdefinestyle{mystyle}{
    backgroundcolor=\color{backcolour},
    commentstyle=\color{codegreen},
    keywordstyle=\color{magenta},
    numberstyle=\tiny\color{codegray},
    stringstyle=\color{codepurple},
    basicstyle=\ttfamily\small,
    breakatwhitespace=false,
    breaklines=true,
    captionpos=b,
    keepspaces=true,
    numbers=left,
    numbersep=5pt,
    showspaces=false,
    showstringspaces=false,
    showtabs=false,
    tabsize=2
}

\begin{document}

\title{Multidimensional specific relative entropy between continuous martingales}

\author[J.\ Backhoff]{Julio Backhoff}

\thanks{
JB acknowledges support by the FWF through projects Y 00782 and P 36835.
}

\author[E.K.\ Bellotto]{Edoardo Kimani Bellotto}

\maketitle

\begin{abstract} 

In continuous time, the laws of martingales tend to be singular to each other. Notably, N.\ Gantert introduced the concept of \emph{specific relative entropy} between real-valued continuous martingales, defined as a scaling limit of finite-dimensional relative entropies, and showed that this quantity is non-trivial despite the aforementioned mutual singularity of martingale laws.

Our main mathematical contribution is to extend this object, originally restricted to one-dimensional martingales, to multiple dimensions. Among other results, we establish that Gantert's inequality, bounding the specific relative entropy with respect to Wiener measure from below by an explicit functional of the quadratic variation, essentially carries over to higher dimensions. We also prove that this lower bound is tight, in the sense that it is the convex lower semicontinuous envelope of the specific relative entropy. This is a novel result even in dimension one. Finally we establish closed-form expressions for the specific relative entropy in simple multidimensional examples.
 
\bigskip


\smallskip

\noindent\emph{Mathematics Subject Classification (2010):} Primary XXX; Secondary XXX.
\end{abstract}

\section{Introduction}
 Recall that if $\mu$ and $\nu$ are probability measures on the same measurable space, then
    the relative entropy between $\mu$ and $\nu$ is defined as
    \[H(\mu|\nu):=\begin{cases}
        \int \log \left(\frac{d\mu}{d\nu}\right )d\mu \, \text{ if } \mu \ll \nu \\
        +\infty \, \text{ otherwise.}
    \end{cases} \]
     This object, although ubiquitous in mathematics, becomes trivial in the case when $\mu$ and $\nu$ are martingale laws in continuous time. Indeed, if e.g.\  $\mu$ is the law of a standard Brownian motion $B$, and $\nu$ is the law of  $2B$, then $\mu$ and $\nu$ are mutually singular (as they can be shown to be concentrated on disjoint sets) and hence their relative entropy is infinite. On the other hand observe that, in this same example, if $\mu_n$ and $\nu_n$ denotes respectively the law of $B$ and $2B$ restricted to times of the form $\{\frac{k}{n}: k=0,\dots,n\}$,  then $\mu_n$ and $\nu_n$ are in fact equivalent to each other and have a finite relative entropy which scales like $n$, as a few computations reveal. This is the motivation behind Gantert's definition of specific relative entropy in \cite{Ga91}, which we now recall. 

    Consider the (filtered) Wiener space $(C,\mathcal F,\{\mathcal{F}_t\}_{t\in [0,1]})$, where $C:=C([0,1])=\{ \omega : [0,1] \rightarrow \R \, |  \, \omega \, \text{ is continuous}\}$, $\mathcal F=\mathcal F_1$, and $\mathcal{F}_t:=\sigma(X_s : s \in [0,t])$ with $X=(X_t)_{t \in [0,1]}$ the canonical process, i.e. $$X_t(\omega)=\omega(t).$$ We define for all $n\in \mathbb{N}$ 
    \[\textstyle \mathcal{F}^n:=\sigma \left(X_{\frac{k}{n}}:k=0,1,\dots,n\right),\]
    and if $\mathbb P,\mathbb Q$ are probability measures on $(C,\mathcal{F})$ we denote by $H(\mathbb{Q}|\mathbb{P})|_{\mathcal{F}^n}$ the relative entropy between the restrictions of $\mathbb{Q}$ and $\mathbb{P}$ to the sub-sigma algebra $\mathcal F^n$. Equivalently, $H(\mathbb{Q}|\mathbb{P})|_{\mathcal{F}^n}$ is the relative entropy between the image measures of $\mathbb{Q}$ and $\mathbb{P}$ by the map
    $C\ni\omega\mapsto (\omega_{k/n})_{k=0}^n$.

With these objects at hand, we can introduce Gantert's specific relative entropy: 
 Let $\mathbb{Q},\mathbb{P} $ be probability measures on $(C,\mathcal{F})$ such that $X$ is a martingale under both of them.  Then the specific relative entropy between $\mathbb{Q}$ and $\mathbb{P}$ is  \[h_1(\mathbb{Q}|\mathbb{P}):=\liminf _{n \rightarrow \infty} 2^{-n}H(\mathbb{Q}|\mathbb{P})|_{\mathcal{F}^{2^n}}\]
 Crucially, this definition comes with an intriguing lower bound. To wit,  in \cite[Satz 1.3]{Ga91} Gantert observed that, if $X$ is a martingale and has an absolutely continuous quadratic variation under $\mathbb Q$ (say, $d\langle X \rangle_t/dt=\sigma^2_t$, $\mathbb Q$-a.s.), then 
 \begin{equation}\label{eq:Gantert_h} \textstyle 
     h_1(\mathbb{Q}|\mathbb{B}^1)\geq \mathbb{E}_{\mathbb Q}\left[\int_0^1 F_1(\sigma^2) dt\right ],
 \end{equation}
 where $\mathbb B^1$ denotes the standard one-dimensional Wiener measure and  
 \begin{equation}
     F_1(x):=\frac{1}{2}\{x-1-\log(x)\}.
 \end{equation}
 The subscript $1$ in $h_1$ resp.\ $F_1$ is there to stress that these objects pertain to one-dimensional martingales resp.\ real numbers.  Inequality \eqref{eq:Gantert_h} has been suitably extended by F\"ollmer to the case where the quadratic variation is not necessarily absolutely continuous \cite{Fo22a,Fo22b}. We will refer to \eqref{eq:Gantert_h} as \emph{Gantert's inequality}. 
 
 It can be argued that the r.h.s.\ of \eqref{eq:Gantert_h} is the most mathematically appealing. For instance, it can be more easily computed than the l.h.s. It can also happen that the l.h.s.\ is infinite while the r.h.s.\ is not; e.g.\ \cite{BeBe23} exhibit this with a continuous martingale which terminates in a Bernoulli distribution. Finally, it can be proved that \eqref{eq:Gantert_h} is an equality in a number of interesting situations (e.g.\ when $\mathbb Q$ is a concrete martingale, such as geometric Brownian motion, or a martingale diffusion with very a regular diffusion coefficient; see \cite{Ga91,BaUn22}). We will add yet another reason in favor of the r.h.s.\ of \eqref{eq:Gantert_h} in our Theorem \ref{thm:hull}, where we show that this lower bound is precisely the convex lower semicontinuous envelope of the l.h.s.\\

 \noindent \textbf{Applications of the specific relative entropy thus far:}
The specific relative entropy $h_1$ between continuous martingales appears in a number of applications in stochastic analysis and finance. 

Gantert \cite{Ga91} shows that $h_1$ is the rate function for the large deviations in a randomized version of Donsker's theorem. In \cite{Fo22a,Fo22b} F\"ollmer generalizes Talagrand's inequality for the Wiener measure to the setting of martingales by using $h_1$. 

In the setting of model calibration in finance, the work \cite{AvFrHoSa97} by Avellaneda et al proposes $h_1$ as a loss function, motivated by considering a sequence of discrete time trinomial models. A similar, but more precise analysis appears in \cite[Ch. 6]{CaGo07} by Cattiaux and Gozlan, involving large deviation techniques. In \cite{BeChLo24}, Benamou, Chazareix and Loeper obtain specific relative entropy minimization problems as a limit to discrete time counterparts. Beyond model calibration, Cohen and Dolinsky \cite{CoDo22} find that $h_1$ appears as a limit of scaled indifference pricing problems involving exponential utility functions. Different problems of entropy  minimization over a class of martingales are \cite{DM18,DMHL19} by De March et al, \cite{He19} by Henry-Labord\`ere, \cite{NuWi24} by Nutz and Wiesel, or \cite{ChCoReWa24} by Chen et al.

In \cite{BeBe23} the authors provide an answer to an open question by D.\ Aldous \cite{Al22a,Al22b}, concerning the binary prediction market which is most random, by minimizing the r.h.s.\ of \eqref{eq:Gantert_h} among all martingales finishing in a given Bernoulli distribution. See also \cite{GuPoRe23} for a concurrent work tackling essentially the same question, but where the $dt$-integral appearing in the r.h.s.\ of \eqref{eq:Gantert_h} runs until an exit time. It is possible to apply scaling limits to other divergences between time-discretized martingales other than the relative entropy; see \cite{BaZh24} for the case of the specific Wasserstein divergence.
\\

We now proceed to describe our contributions.

 \subsection{Multidimensional specific relative entropy}

 We identify the multivariate / multidimensional analogue to the specific relative entropy. Hence $X$ denotes now the canonical process on the space of continuous $\R^l$-valued paths. Suppose $\mathbb Q,\mathbb P$ are given, and such that $X$ is a martingale under both of these probability measures. We define their specific relative entropy just as in the one-dimensional case, i.e.
 \[h_l(\mathbb{Q}|\mathbb{P}):=\liminf _{n \rightarrow \infty} 2^{-n}H(\mathbb{Q}|\mathbb{P})|_{\mathcal{F}^{2^n}}.\]
 Our first result is to obtain the analogue to the bound \eqref{eq:Gantert_h} in multiple dimensions. Henceforth we denote by $\mathbb B^l$ the standard Wiener measure in dimension $l$, i.e.\ the law of standard $\R^l$-valued Brownian motion, and \begin{equation}
     F_l(\Sigma):=\frac{1}{2}\{\text{tr}(\Sigma)-l-\log(\text{det}(\Sigma))\},
 \end{equation}
 with tr denoting trace of a matrix and det denoting determinant.
 \begin{thm}[Gantert's inequality in $\mathbb  R^l$]\label{thm:gantert_ineq}
     Suppose $X$ is a square integrable $\mathbb Q$-martingale admitting an absolutely continuous quadratic covariation $d\langle X\rangle_t/dt = \Sigma^{\mathbb Q}_t$ and $X_0=0$ ($\mathbb Q$-a.s.). Then
      \begin{equation}\label{eq:Gantert_h_Rl} \textstyle 
     h_{l}(\mathbb{Q}|\mathbb{B}^l)\geq \mathbb{E}_{\mathbb Q}\left[\int_0^1 F_l(\Sigma_t^{\mathbb Q}) dt\right ].
 \end{equation}
 \end{thm}
 We remark that the inequality is still true if $\mathbb Q$ has a general starting distribution, as long as $\mathbb{B}^l$ is replaced by the law of $\R^l$-valued Brownian motion with the same starting law. The proof of Theorem \ref{thm:gantert_ineq}, which follows closely the one-dimensional arguments in \cite[Satz 1.3]{Ga91}, is provided in Section \ref{sec:multidim_spec_ent}. We also explore a number of elementary properties of the multidimensional specific relative entropy. Notably,
 \begin{itemize}
     \item In Proposition \ref{thm2.2} we show that \eqref{eq:Gantert_h_Rl} is an equality in the case that $\mathbb Q$ is the law of a Gaussian integral. This is an extension of \cite[Satz 1.2]{Ga91} in the one-dimensional case.
     \item The tensorized lower bound (Lemma \ref{lem:tens_lb}):
     \[ \textstyle  h_l(\mathbb Q| \mathbb P^1\otimes \dots\otimes \mathbb P^l ) \geq \sum_{i=1}^l h_1(\mathbb Q^i|\mathbb P^i),\]
where $\mathbb Q^i$ is the projection of $\mathbb Q$ into its $i$-th spatial component, and each of the $\mathbb P^i$ is a one-dimensional martingale law. This result is in tandem with a tensorization inequality for relative entropies.
 \end{itemize}

\subsection{Tightness of Gantert's inequality}

Our next main result, Theorem \ref{thm:hull}, seems to have no precursor in the literature (even in dimension one). We show that Inequality \eqref{eq:Gantert_h_Rl} is sharp, in the sense that the lower bound therein is maximal among all lower bounds that are proper, convex, lower semicontinuous functions of $\mathbb Q$:
 
 \begin{thm}\label{thm:hull}
     The functional $$\textstyle \mathbb Q\mapsto \mathbb{E}_{\mathbb Q}\left[\int_0^1 F_l(\Sigma^{\mathbb Q}_t) dt\right ],$$
 is the convex lower semicontinuous envelope of the functional $$\mathbb Q\mapsto h_{l}(\mathbb{Q}|\mathbb{B}^l),$$ i.e.\ the former is the largest convex lower semicontinuous minorant of the latter.
 \end{thm}

It should be stressed that  $\mathbb Q\mapsto \mathbb{E}_{\mathbb Q}\left[\int_0^1 F_l(\Sigma^{\mathbb Q}_t) dt\right ]$ is an affine function. Indeed, one can identify an adapted matrix-valued process $\hat \Sigma$ such that, for each $\mathbb Q$ as in Theorem \ref{thm:gantert_ineq}, we have $\mathbb Q(\text{Leb}\{t\in[0,1]:\Sigma_t^{\mathbb Q}\neq\hat \Sigma_t\})=0$; see Remark \ref{rem:Karandakar}. 

In the authors' opinion, the previous theorem is the main reason why the functional in the r.h.s.\ of Gantert's inequality is the superior mathematical object. This point of view has been also defended and put to practice in \cite{BeBe23} in order to determine the most random prediction market. The latter pertains to a binary prediction market, i.e.\ where at a final date a binary outcome is revealed. A generalization to higher dimensions (corresponding to prediction markets where multiple outcomes can take place at a terminal time) is being carried out in \cite{BaWaZh24}, based on the idea of minimizing the r.h.s.\ of \eqref{eq:Gantert_h_Rl} over a class of suitable multidimensional martingales.

 \subsection{Closed-form expressions}

 For mathematical finance, and beyond, it is desirable to explicitly determine the specific relative entropy of a Black-Scholes model (i.e.\ geometric Brownian motion) with respect to Brownian motion, or between two Black-Scholes models with different parameters. Indeed, the lack of closed-form formulae in these cases would arguably make the concept of specific relative entropy irrelevant for some if not most applications. As it turns out, the multidimensional specific relative entropy takes an explicit form in these examples.

 We recall that a martingale Black-Scholes model $M^{\Gamma}$ is the solution to the SDE
 $dM_t^{\Gamma}:=diag(M_t^{\Gamma})\Gamma dB_t$, equivalently  $dM_t^i:=M_t^i(\sum_{k=1}^l\Gamma_{ik} d B_t^k)$ for all $i=1,\dots,l$, with $B$ a standard $\R^l$-valued Brownian motion. We take $\Gamma$ to be a constant matrix; an extension to time-dependent $\Gamma_t$ would be achievable. 
\begin{lem}\label{lem:Black-Scholes_vs_BM_intro}
{The specific relative entropy between a $\R^l$-valued martingale Black-Scholes models with parameter $\Gamma\in \mathbb{R}^{l \times l}$ and Brownian motion, both starting at the same point, is given by:
\begin{equation} \begin{split} \label{srE_GBM_BM_l} 
h_l(\mathcal L(M^{\Gamma})|\mathbb{B}^l)=\textstyle \frac{1}{2}\Big(\sum_{i=1}^l \big(e^{\sum_{k=1}^l \Gamma_{ik}^2}-1\big)-l+\frac{1}{2} tr(\Gamma \Gamma^T) - \log \big(\det (\Gamma \Gamma^T) \big)\Big).
\end{split}
\end{equation}}
\end{lem}

 \begin{lem}\label{lem:Black-Scholes}
The specific relative entropy between two $\R^l$-valued martingale Black-Scholes models with parameters $\Gamma_1,\Gamma_2 \in \mathbb{R}^{l \times l}$, starting at the same point, is:
{\begin{equation}\label{equazioneB}
h_l(\mathcal L(M^{\Gamma_1})|\mathcal L(M^{\Gamma_2}))=\textstyle  \frac{1}{2}\Big(tr\big((\Gamma_2 \Gamma_2^T)^{-1}(\Gamma_1 \Gamma_1^T)\big)-l - \log \frac{\det (\Gamma_1 \Gamma_1^T)}{\det (\Gamma_2 \Gamma_2^T)}\Big) 
\end {equation}}
\end{lem}

One important message coming from these lemmas is that Gantert's inequality \eqref{eq:Gantert_h_Rl} becomes an actual equality in these cases, just as in the one-dimensional situation. See Remarks \ref{rem:eq_first_ex} and \ref{rem:eq_second_ex}.

We refer the reader to \cite{Be24} for numerical simulation approximating the specific relative entropy (more precisely, its lower bound) between certain stochastic volatility models, such as SABR and Heston models, and Brownian motion. In these cases no closed-form expressions are known.

\section{The multidimensional specific relative entropy}
\label{sec:multidim_spec_ent}

Throughout $l\in\mathbb N$. We recall that the trace of a square matrix $A \in \R^{l \times l}$ is the sum of the diagonal elements of $A$, i.e. $tr(A):=\sum_{i=1}^l A_{ii}$. Furthermore, the trace is a linear map and it is commutative under the product of matrices. We will denote by $A^T$ the transpose of matrix $A$. Notice that if $x \in \R^l$ then $tr(Axx^T)=x^TAx$. We denote the $l\times l$-identity matrix with the symbol $\mathcal{I}$, the $l\times l$-null matrix with the symbol $\mathcal{O}_{l\times l}$, the $l$-null vector with the symbol $0_l$ {and the $l$-vector of all ones with $\mathbf{1}$.} Whenever $A$ is symmetric positive definite, we write $A^{1/2}$ for its unique symmetric positive definite square root.
If $Y$ is a random variable, we write $\mathcal L(Y)$ for its law. If the underlying probability space $(\Omega,\mathcal F,\mathbb P)$  is not obviously fixed from the context, we write $\mathcal L_{\mathbb P}(Y)$.

\subsection{Preliminaries}

As we recall in Appendix \ref{app:A}, for  two normal variables in $\mathbb R^l$ that have the same mean, we have:
    \begin{equation}\label{eq:prelim_ent_Gaussians} \textstyle 
        H(\mathcal{N}_l(\mu, \Sigma_1)|\mathcal{N}_l(\mu, \Sigma_2))=\frac{1}{2}\Big(tr(\Sigma_2^{-1}\Sigma_1)-l-\log\left( \frac{\det \Sigma_1}{\det \Sigma_2}\right)\Big).
    \end{equation}
 This suggests the definition: 
 
 \begin{defn}
    We write $$F_l(X):=\frac{1}{2}(tr(X)-l-\log(\det X)),$$ for every symmetric positive definite matrix $X\in \R^{l \times l}$. Thus 
\begin{align*}
  H(\mathcal{N}_l(\mu, \Sigma_1)|\mathcal{N}_l(\mu, \Sigma_2))
  & = \textstyle  \frac{1}{2}\Big(tr(\Sigma_2^{-1}\Sigma_1)-l-\log( \frac{\det \Sigma_1}{\det \Sigma_2})\Big) \\
  &=\textstyle  \frac{1}{2}\Big(tr(\Sigma_1^{1/2}\Sigma_2^{-1}\Sigma_1^{1/2})-l-\log( \det (\Sigma_1^{1/2}\Sigma_2^{-1}\Sigma_1^{1/2}))\Big) \\
  &=F_l(\Sigma_1^{1/2}\Sigma_2^{-1}\Sigma_1^{1/2})\\
  &=F_l(\Sigma_2^{-1/2}\Sigma_1\Sigma_2^{-1/2}) .
    \end{align*}
    With some abuse of notation we still write $F(\Sigma_2^{-1}\Sigma_1)$ for the above number, even if $\Sigma_2^{-1}\Sigma_1$ need not be symmetric.
    \end{defn}

We start observing some elementary properties concerning this definition.  We defer the proof of these facts to Appendix \ref{app:B}.
  
\begin{lem}
\label{F_l}
\begin{enumerate}
    \item  Let $\alpha \in \mathbb{R}_+$. then $F_l(\alpha \mathcal{I})=l F_1(\alpha)$.
    \item  Let $D=diag(\sigma_i^2)$ be a diagonal matrix with $\sigma_i^2 \in \mathbb{R}_+$ ($i=1, \dots, l$) on the diagonal. Then $F_l(D)=\sum _{i=1}^lF_1(\sigma_i^2)$.
    \item  Let $\Sigma \in \mathbb{R}^{l \times l}$ symmetric and positive definite, and let $D=A^{-1}\Sigma A=diag(eigenvalues\, of \, \Sigma)$. Then $F_l(\Sigma)=F_l(D)=\sum_{\sigma^2\in eigenvalues\, of \, \Sigma} F_1(\sigma^2)$. 
    \item 
    The function $X\mapsto F_l(X):= \frac{1}{2}(tr(X)-l-\log(\det X))$, from the set of $l\times l$ symmetric positive semidefinite matrices into the reals, is convex, non-negative, and it vanishes exclusively at $X=\mathcal I$.
    \end{enumerate}
\end{lem}

For later use the next remarks will be useful:

\begin{rmk} \label{ex_extra} {The relative entropy at time $t$ between a $l$-multivariate martingale Black-Scholes model $M_t^{\Gamma}$ with parameter $\Gamma \in \R^{l\times l}$, and a $l$-multivariate Brownian motion $B_t^l$, both starting at $\mathbf{1}$, is: \begin{equation*} \begin{split} 
&\textstyle H(\mathcal{L}(M_t^{\Gamma})| \mathcal{L}(B_t^l))=\frac{1}{2}\Big(-\log (\det (\Gamma \Gamma^T) ) + t \sum_{i=1}^l \sum_{k=1}^l \Gamma_{ik}^2 -l + \sum_{i=1}^l \frac{e^{t\sum
_{k=1}^l\Gamma_{ik}^2}-1 }{t} \Big).
\end{split}
\end{equation*}}
{We defer to Appendix \ref{app:B} for a proof.}
\end{rmk}

\begin{rmk}
\label{ex1.4}
     Given two $l$-multivariate lognormal distributions $ \sim lognormal(\mu_i,\Sigma_i)$ where $i=1,2$, $\mu_i \in \mathbb{R}^l$, $\Sigma_i \in \mathbb{R}^{l \times l}$, then the relative entropy between the two variables coincides with the relative entropy between two $l$-multivariate normal variables with the same parameters, i.e.: 
    \begin{equation*} \begin{split} 
    H(lognormal(\mu_1,\Sigma_1)|lognormal(\mu_2,\Sigma_2))=H(\mathcal{N}_l(\mu_1,\Sigma_1)|\mathcal{N}_l(\mu_2,\Sigma_2)).
    \end{split} \end{equation*}
Indeed, this is the consequence of a more general result, given by Lemma \ref{lem1.2} in the appendix with the choice $T(x)=e^x$. 
\end{rmk} 

\subsection{Wiener space setting and Gaussian computations}

We consider the Wiener space $(C,\mathcal{F})$ of $\R^l-$valued continuous paths, so $$C:=C([0,1];\R^l):=\{ \omega : [0,1] \rightarrow \R^l \, |  \, \omega \, \text{ is continuous}\},$$ and $\mathcal{F}:=\sigma(X_t : t \in [0,1])$, where $X=(X_t)_{t \in [0,1]}$ is the canonical process, i.e. $X_t(\omega)=\omega(t)$. Further we consider the filtration $\mathcal{F}_t=\sigma(X_s:s \leq t)$. 

Since we are interested in the laws of continuous martingales, we introduce:
\begin{defn}
  $\mathcal{M}_1(C)$ is the collection of all probability measures on $(C,\mathcal{F})$. A martingale measure is a probability measure $\Q \in \mathcal{M}_1(C) $ such that the canonical process X is a martingale with respect to the filtration $(\mathcal{F}_t)_{t \in [0,1]}$ under $\Q$. A martingale measure is called square integrable, denoted by $\mathbb Q\in \mathcal{M}^2_l$, if $\mathbb E_{\mathbb Q}[|X^i_1|^2]<\infty$ for $i=1,\dots, l$. 
\end{defn}

We will denote throughout by 
$$\mathbb{B}:= \mathbb{B}^l\in \mathcal{M}^2_l$$ the law of the standard Brownian motion in $\R^l$, i.e.\ the Wiener measure.

\begin{rmk}
    The relative entropy as a measure of ``distance" or ``discrepancy'' is only interesting when there is absolute continuity. Hence, it is not a useful tool for continuous martingale measures; see Example \ref{ex:two_BM} for an explicit computation.  To go over this case, one defines for all $n\in \mathbb{N}$ 
    \[\mathcal{F}^n:=\sigma (X_{k/n}:k=0,1,\dots,n).\]
    In this way, it might happen that $\mathbb{Q}$  is not absolutely continuous to $\mathbb{P}$ in the whole sigma-algebra $\mathcal{F}$, while it is in the smaller sigma-algebras $\mathcal{F}^n$ for all $n \in \mathbb{N}$. 
\end{rmk}

\begin{defn}
 Let $n \in \mathbb{N}$ and $\mathbb{P},\mathbb{Q} \in \mathcal{M}_1(\mathcal{C})$ martingale measures such that $\mathbb{Q}|_{\mathcal{F}^n}\ll \mathbb{P}|_{\mathcal{F}^n}$. 
We denote the Radon-Nikodym density of $\mathbb{Q}|_{\mathcal{F}^n}$ with respect to $\mathbb{P}|_{\mathcal{F}^n}$ on $(\mathcal{C},\mathcal{F}^n)$ as $\frac{d\mathbb{Q}}{d\mathbb{P}}|_{\mathcal{F}^n}$ and define: 
\begin{equation*} \begin{split}
H(\mathbb{Q}|\mathbb{P})|_{\mathcal{F}^n}:=\begin{cases}
    \int _{\mathcal{C}} \log \frac{d\mathbb{Q}}{d\mathbb{P}}|_{\mathcal{F}^n}d\mathbb{Q} \, \text{ if } \, \mathbb{Q}|_{\mathcal{F}^n} \ll \mathbb{P}|_{\mathcal{F}^n} \\
    +\infty \, \text{ otherwise.}
\end{cases}
\end{split} \end{equation*}
\end{defn}
Let $n\in \N$. For $k=1,\dots,n$, we use the following notation. For the $k$-th increment of a process $Y$ we write $$\Delta_k^n Y:= Y_{\frac{k}{n}}-Y_{\frac{k-1}{n}},$$ while for a function  $\xi $ we use the notation $\Delta_k^n\xi := \xi(\frac{k}{n})-\xi(\frac{k-1}{n})$.

The following useful lemma is immediate: 
\begin{lem} \label{Lemma2.1}
    Let $n \in \mathbb{N}$ and $\mathbb{P},\mathbb{Q}\in \mathcal{M}_1(C) $ such that $\mathbb{Q}|_{\mathcal{F}^n}\ll \mathbb{P}|_{\mathcal{F}^n}$. Then $\mathbb{P}$-a.s., 
    \begin{equation*}
        \frac{d\mathbb{Q}}{d\mathbb{P}}|_{\mathcal{F}^n}(\omega)=\frac{d\mathcal{L}_{\mathbb{Q}}(X_0, \Delta_1^n X,\dots,\Delta_n^nX)}{d\mathcal{L}_{\mathbb{P}}(X_0, \Delta_1^n X,\dots,\Delta_n^nX)}(\omega(0),\Delta_1^n\omega, \dots, \Delta_n^n\omega).
    \end{equation*}
    In particular, 
    \begin{equation*}
H(\mathbb{Q}|\mathbb{P})|_{\mathcal{F}^n}=H(\mathcal{L}_{\mathbb{Q}}(X_0, \Delta_1^n X,\dots,\Delta _n ^n X)| \mathcal{L}_{\mathbb{P}}(X_0, \Delta_1^n X,\dots,\Delta_n^nX)).
    \end{equation*}
    Moreover, if $\mathbb{P}$ and $\mathbb{Q}$ are such that $X$ has independent increments. Then
    \[H(\mathbb{Q}|\mathbb{P})|_{\mathcal{F}^n}=H(\mathcal{L}_{\mathbb{Q}}(X_0)|\mathcal{L}_{\mathbb{P}}(X_0))+\sum _{k=1}^n H(\mathcal{L}_{\mathbb{Q}}(\Delta _k^n X)|\mathcal{L}_{\mathbb{P}}(\Delta _k^n X)).\]
\end{lem}

\begin{rmk}
 Let $\mathbb{Q},\mathbb{P} \in \mathcal{M}_1(C)$ be martingale measures such that $X_0=x$ for some $x \in \mathbb{R}^l$ both $\mathbb{Q}$-a.s.\ and $\mathbb{P}$-a.s. Then we can disregard the slot for time $0$ when computing the relative entropy on the restrictions to $\mathcal{F}^n$. Indeed, we have $\mathbb{P}$-a.s., 
    \begin{equation*} \begin{split}
    &\frac{d\mathcal{L}_{\mathbb{Q}}(X_0, \Delta_1^n X,\dots,\Delta_n^nX)}{d\mathcal{L}_{\mathbb{P}}(X_0, \Delta_1^n X,\dots,\Delta_n^nX)}(\omega(0),\Delta_1^n\omega, \dots, \Delta_n^n\omega)\\=&\frac{d\mathcal{L}_{\mathbb{Q}}( \Delta_1^n X,\dots,\Delta_n^nX)}{d\mathcal{L}_{\mathbb{P}}( \Delta_1^n X,\dots,\Delta_n^nX)}(\Delta_1^n \omega, \dots, \Delta_n^n\omega),
\end{split}\end{equation*}
and therefore $H(\mathbb{Q}|\mathbb{P})|_{\mathcal{F}^n}=H(\mathcal{L}_{\mathbb{Q}}( \Delta_1^n X,\dots,\Delta _n ^n X)| \mathcal{L}_{\mathbb{P}}( \Delta_1^nX,\dots,\Delta_n^nX)).$
\end{rmk}

\begin{defn}
   Let $\mathbb{Q},\mathbb{P} \in \mathcal{M}_1(C)$  be martingale measures.  Then the specific relative entropy of $\mathbb{Q}$ with respect to $\mathbb{P}$ is defined as \[h_l(\mathbb{Q}|\mathbb{P}):=\liminf _{n \rightarrow \infty}2^{-n}H(\mathbb{Q}|\mathbb{P})|_{\mathcal{F}^{2^n}}.\]
\end{defn}

\begin{ex}\label{ex:two_BM}
By Lemma \ref{Lemma2.1} we can compute the specific relative entropy between two $l$-dimensional Brownian motions: $\mathbb{Q}=\mathcal{L}((\sigma B_t)_{t \in [0,1]})$ and  $\mathbb{P}=\mathcal{L}((\eta B_t)_{t \in [0,1]})$ with $\sigma, \eta \in \mathbb{R}$.
Indeed,
\begin{equation*} \begin{split}
    H(\mathbb{Q}|\mathbb{P})|_{\mathcal{F}^n}=\sum_{k=1}^{n} H(\mathcal{L}(\Delta _k^n \sigma B)|\mathcal{L} (\Delta _k^n \eta B)).
    \end{split}\end{equation*}
Now, given the law of the increments:
    \[ \textstyle \sigma \mathcal{N}_l\left (0_l, \frac{k}{n} \mathcal{I}\right)-\sigma \mathcal{N}_l\left(0_l, \frac{k-1}{n} \mathcal{I}\right) \sim \mathcal{N}_l\left(0_l, \frac{\sigma^2}{n} \mathcal{I}\right).\]
Then we have
    \[\textstyle H(\mathbb{Q}|\mathbb{P})|_{\mathcal{F}^n}=\sum_{i=1}^n H\left(\mathcal{N}_l\left(0_l, \frac{\sigma^2}{n} \mathcal{I}\right)| \mathcal{N}_l\left(0_l, \frac{\eta^2}{n} \mathcal{I}\right)\right),\]
    so by Remark \ref{rmk2} in the appendix, this number is equal to:
    \begin{equation*} \begin{split}
      \textstyle  \sum_{i=1}^n F_l\big((\frac{\eta ^2}{n} \mathcal{I})^{-1}( \frac{\sigma ^2}{n} \mathcal{I})\big) &= \textstyle  \sum_{i=1}^n F_l\big(\frac{\sigma^2}{\eta^2}\mathcal{I}\big) 
    =\frac{1}{2} \sum_{k=1}^n tr(\frac{\sigma ^2}{\eta^2}\mathcal{I})-l-\log \big(\det \frac{\sigma^2}{\eta^2}\mathcal{I}\big)\\
    &=\textstyle \frac{1}{2} \sum_{k=1}^n l \frac{\sigma ^2}{\eta^2} -l-l \log(\frac{\sigma ^2}{\eta^2})    = n \frac{l}{2}(\frac{\sigma ^2}{\eta^2} -1-\log \frac{\sigma ^2}{\eta^2}).
    \end{split}\end{equation*}
    As this sequence explodes with $n$ if $\sigma^2\neq\eta^2$, this already shows that $H(\mathbb Q|\mathbb P)=+\infty$. 
    Finally,  by definition of the specific relative entropy:
    \begin{equation*} \begin{split}h_l(\mathbb{Q}|\mathbb{P})&= \textstyle  \liminf_{n \rightarrow \infty} 2^{-n} H(\mathbb{Q}|\mathbb{P})|_{\mathcal{F}^{2^n}}=\frac{l}{2}\left(\frac{\sigma ^2}{\eta^2} -1-\log \frac{\sigma ^2}{\eta^2}\right),
    \end{split}\end{equation*}
    and the liminf is a limit, even if we considered integers that are not of the form $2^n$.
\end{ex}

We now turn our attention to the important case of general Gaussian martingales: Let $a(t) \in \R^{l \times l}$ be a differentiable symmetric positive definite matrix-valued function such that for times $t\geq s$ we have $a(t)\geq a(s)$ in the sense of positive matrices. Then there exists a  symmetric and positive definite matrix $G(s)\in \R^{l \times l} $ such that $a(t)=\int _0^t G(s)ds$. 

\begin{rmk}
    Let $B_t$ be the standard Brownian motion in $\R^l$ and let $G(t) \in \R^{l \times l}$ as above. Define $\bar{B}_t:= \int _0^1 G(s)^{1/2} dB_s$, i.e.\ for $i=1,\dots,l$ the coordinates of $\bar B$ are defined as $\bar{B}^i_t:=\int _0^t \sum _{k=1}^l (G(s)^{1/2})_{ik}dB_s^k$. We have:
    \begin{equation*}\begin{split}
        \mathbb{E}[\bar{B}^i_t \bar{B}^j_t]=& \textstyle \mathbb{E}\Big[\sum_u\sum_v \int _0^t (G(s)^{1/2})_{iu}dB_s^u \int _0^t (G(s)^{1/2})_{jv}dB_s^v\Big]\\
  =& \textstyle \sum_u \mathbb{E}\Big[\int _0^t (G(s)^{1/2})_{iu}dB_s^u \int _0^t (G(s)^{1/2})_{ju}dB_s^u\Big]\\
=& \textstyle \int _0^t \sum _u (G(s)^{1/2})_{iu}(G(s)^{1/2})_{ju}ds
=\int _0^t (G(s)^{1/2}(G(s)^{1/2})^T)_{ij}ds
\\ =&\textstyle \int _0^t (G(s))_{ij}ds = a_{ij}(t).
\end{split} \end{equation*}
\end{rmk}
Thus, for every $a(\cdot)$ as above, we can define  the martingale measure $\mathbb{Q}^a:=\mathcal{L}(\bar{B})= \mathcal{L}( \int _0^\cdot G(s)^{1/2} dB_s)$, and if we further assume that $\forall i=1,\dots, l$, $\int_0^1 G_{ii}(s)ds< \infty$, then  $\mathbb{Q}^a \in \mathcal{M}^2_l$. By the previous remark, the time $t$ marginal of $\mathbb{Q}^a$ is equal to $\mathcal{N}_l(0_l, a(t))$, and $\mathbb{Q}^a$ has  independent increments.  
Now we can see what \cite[Satz 1.2]{Ga91} becomes in this multivariate setup: 
\begin{prop} \label{thm2.2}
    We have:
\begin{equation*}\begin{split}  
h_l(\mathbb{Q}^a|\mathbb{B})&=\textstyle \frac{1}{2}\int _0^1\big(
tr(G(s)) -l -\log (\det G(s))\big)ds =\sup _{n\in \N} \frac{1}{n} H(\mathbb{Q}^a|\mathbb B)|_{\mathcal{F}^n} \in [0, \infty],
\end{split} \end{equation*}
and $\frac{1}{n}\log(\frac{d\mathbb{Q}^a}{d\mathbb{B}})|_{\mathcal{F}^n}\xrightarrow{L^1(\mathbb{Q}^a)} \frac{1}{2}\int _0^1 \big( tr(G(s))-l-\log(\det G(s)) \big) ds$ if the latter is finite.
\end{prop}

\begin{proof} We proceed in parallel with the proof of \cite[Satz 1.2]{Ga91}.
    Under $\mathbb{Q}^a$ the process $X$ is $l$-Gaussian with independent increments and for all $s<t$, 
\begin{equation*}
X_t-X_s \sim_{\mathbb{Q}^a} \mathcal{N}_l(0_l,a(t)-a(s)),\,\,\textstyle  B_{\frac{k}{n}}-B_{\frac{k-1}{n}} \sim \mathcal{N}_l\left ( 0_l,\frac{1}{n}\mathcal{I}\right ).
\end{equation*}
Recall that $H(\mathcal{N}_l(0_l, \Sigma_{1})|\mathcal{N}_l(0_l, \Sigma_{2}))=F_l(\Sigma_{2}^{-1}\Sigma_{1})$.
By Lemma \ref{Lemma2.1} and independence of increments under $\mathbb{Q}^a$ as well as under $\mathbb{B}$, we can write:
\begin{equation} \label{eq11}\begin{split}
H(\mathbb{Q}^a|\mathbb{B})|_{\mathcal{F}^n}&=\textstyle  \sum _{k=1}^n H(\mathcal{L}_{\mathbb{Q}^a}(\Delta_k^nX)|\mathcal{L}_{B}(\Delta_k^nX))\\
&=\textstyle \sum _{k=1}^n  H\left(\mathcal{N}_l(0_l,\Delta_k^n a|\mathcal{N}_l\left(0_l, \frac{1}{n}\mathcal{I}\right)\right)\\
&=\textstyle \sum _{k=1}^n F_l\left((\mathcal{I}/n)^{-1} \Delta_k^n a\right).
\end{split}\end{equation}
Now let $\mathcal{B}_n:=\sigma \big((\frac{k-1}{n},\frac{k}{n}]: k=1, \dots, n\big)$. Since we have $a(t)=\int _0^t G(s)ds$, then
\begin{equation*}\begin{split} \textstyle 
\mathbb{E}_{\lambda}[G|\mathcal{B}_n](t)=n \sum _{k=1}^n \Delta_k^n a \mathbbm{1}_{(\frac{k-1}{n},\frac{k}{n}]}(t). 
\end{split}\end{equation*}
with $\lambda$ the Lebesgue measure on $[0,1]$. Moreover, we get in \eqref{eq11} that
\begin{equation*}\begin{split} \textstyle 
\sum _{k=1}^n F_l(n\Delta_k^n a)&=n\int_0^1 F_l(n\Delta_k^n a\mathbbm{1}_{(\frac{k-1}{n},\frac{k}{n}]}(s))ds=n\int_0^1 F_l\big(\mathbb{E}_{\lambda}[G|\mathcal{B}_n](s)\big)ds,
\end{split}\end{equation*}
and dividing by $n$ we obtain
\begin{equation*}\begin{split} \textstyle 
\frac{1}{n}H(\mathbb{Q}^a|\mathbb{B})|_{\mathcal{F}^n}=\int_0^1 F_l(\mathbb{E}_{\lambda}[G|\mathcal{B}_n](s))ds.
\end{split}\end{equation*}
We now use that the function F is convex (Lemma \ref{F_l}), so by Jensen’s conditional inequality and the tower property  we obtain 
\begin{equation*}\begin{split}\textstyle 
\frac{1}{n}H(\mathbb{Q}^a|\mathbb{B})|_{\mathcal{F}^n} &=\textstyle \int _0^1F_l(\mathbb{E}_{\lambda}[G|\mathcal{B}_n](s)) ds 
\leq \int _0^1 \mathbb{E}_{\lambda} [F_l(G)|\mathcal{B}_n](s) ds
= \int _0^1 F_l(G(s))ds. \end{split}\end{equation*}
Note that $\mathbb{E}_{\lambda}[G|\mathcal{B}_n] \rightarrow G$ $\lambda$-a.e, by the Lebesgue differentiation theorem. Using
the a.s.\ convergence we now get: 
\begin{equation*}\begin{split}
\textstyle \int_0^1F_l(G(s))ds&=\textstyle  \int _0^1 \liminf_{n \rightarrow \infty} F_l(\mathbb{E}_{\lambda}[G|\mathcal{B}_n](s)) ds 
 \leq \liminf_{n \rightarrow \infty} \int_0^1  F_l(\mathbb{E}_{\lambda}[G|\mathcal{B}_n](s)) ds \\& \textstyle  \leq \limsup_{n \rightarrow \infty}\int_0^1  F_l(\mathbb{E}_{\lambda}[G|\mathcal{B}_n](s)) ds   \leq \int _0^1 F_l(G(s))ds, 
\end{split}\end{equation*}
where the first inequality is by Fatou’s Lemma and $F_l \geq 0$ ((Lemma \ref{F_l})) while the last inequality is by the upper bound obtained by Jensen above. We sum up,
\begin{equation*}\begin{split}\textstyle 
h(\mathbb{Q}^a|\mathbb{B})&=\textstyle \lim_{n \rightarrow \infty} \frac{1}{n} H(\mathbb{Q}^a|\mathbb{B})|_{\mathcal{F}^n} 
=\lim _{n \rightarrow \infty} \int _0^1
F_l(\mathbb{E}_{\lambda}[G|\mathcal{B}_n](s)) ds 
=\sup _{n\geq 1} \frac{1}{n}  H(\mathbb{Q}^a|\mathbb{B})|_{\mathcal{F}^n} \\&\textstyle  =\int_0^1F_l(G(s))ds,
\end{split}\end{equation*}
which proves the first part of the theorem. \\ 

Now we assume that $ \int_0^1 (tr(G(s))-l-\log\det G(s)) ds <\infty$ and in particular $\Delta_k^n a >0$ for all $n \in \N$ and $k=1,\dots,n$. To show $L^1$ convergence, note that applying Lemma \ref{Lemma2.1} plus some linear algebra: 
\begin{equation}\label{eq12}\begin{split} \textstyle 
& \textstyle \frac{1}{n}\log \frac{d\mathbb{Q}^a}{d\mathbb{B}}|_{\mathcal{F}^n}(X)= \frac{1}{n}\log \Pi_k \frac{d \mathcal{N}_l(0_l,\Delta_k^n a )}{d \mathcal{N}_l(0_l, \frac{1}{n}\mathcal{I})}(\Delta_k^nX)\\
&=\textstyle \frac{1}{2n} \sum _{k=1}^n \big(-\log (\det \Delta_k^n a) -(\Delta_k^nX - 0_l)^T \Delta_k^n a^{-1} (\Delta_k^nX - 0_l)+\\& +\log (\det \frac{1}{n}\mathcal{I}) + (\Delta_k^nX - 0_l)^T (\frac{1}{n}\mathcal{I})^{-1}  (\Delta_k^nX - 0_l)^T \big)
\\&\textstyle =\frac{1}{2n} \sum _{k=1}^n tr(n\mathcal{I}\Delta_k^nX (\Delta_k^nX )^T )-tr((\Delta_k^n a^{-1}\Delta_k^nX (\Delta_k^nX )^T ) -log (n^l \det\Delta_k^n a )\\
&= \textstyle tr(\frac{1}{2n} \sum_{k=1}^n n\mathcal{I}\Delta_k^nX (\Delta_k^nX )^T  )- tr(\frac{1}{2n} \sum_{k=1}^n (\Delta_k^n a^{-1}\Delta_k^nX (\Delta_k^nX )^T ) +\\& \textstyle -\frac{1}{2n}\sum_{k=1}^n log (n^l \det\Delta_k^n a ). 
\end{split}\end{equation}
As in the one dimensional case we get the $L^2(\mathbb Q^a)$ limit
\begin{equation*}\begin{split}\textstyle 
\frac{1}{n} \sum _{k=1}^n n\mathcal{I}\Delta_k^nX (\Delta_k^nX )^T \xrightarrow{n\rightarrow \infty} a(1)=\int _0^1 G(s)ds,
\end{split}\end{equation*}
and the $L^1(\mathbb Q^a)$ limit
\begin{equation*}\begin{split}\textstyle 
\frac{1}{n} \sum_{k=1}^n \Delta_k^n a^{-1}\Delta_k^nX (\Delta_k^nX )^T \xrightarrow{n\rightarrow \infty } \mathcal{I}.
\end{split}\end{equation*}
For the last term, we note  from the first part of the proof, that
\begin{equation*}\begin{split}
&\textstyle \lim _{n\rightarrow \infty} \mathbb{E}_{\mathbb{Q}^a}\Big[\frac{1}{2n} \sum _{k=1}^n tr(n\mathcal{I}\Delta_k^nX (\Delta_k^nX )^T )-tr((\Delta_k^n a^{-1}\Delta_k^nX (\Delta_k^nX )^T ) +\\ &\textstyle -\log (n^l \det\Delta_k^n a )\Big]= \int_0^1 F_l(G(s))ds  =\frac{1}{2}\int_0^1 tr(G(s))-l-\log(\det G(s)) ds.
\end{split}\end{equation*}
Therefore,  we have obtained the deterministic limit: 
\begin{equation*}\begin{split} \textstyle 
-log (n^l \det\Delta_k^n a ) \xrightarrow{n\rightarrow \infty } \int _0^1 -\log (\det G(s))ds. 
\end{split}\end{equation*}
\end{proof}

In the above proof it was not necessary to consider integers of the form $2^n$ in the definition of the specific relative entropy. 

Before moving any further, we recall the notion of quadratic covariation:
 \begin{defn}
 Let $T>0$, $t \in [0,T]$ and $l \in \N$. The quadratic covariation for a multivariate martingale $X_t^l=(X_t^1,\dots,X_t^l) $ is the matrix-valued process $ \bigl\langle X \bigr \rangle $ with entries  $ \bigl\langle X^i,X^j\bigr \rangle_t $ with $i,j=1,\dots,l$ and  \[\textstyle \bigl\langle X^i,X^j\bigr \rangle_t=\lim_{|\Pi|\rightarrow 0} \sum _{k=1} (X_{t_k}^i-X_{t_{k-1}}^i)(X_{t_k}^j-X_{t_{k-1}}^j),\]
 where $\Pi=\{0=t_0<t_1<\dots < t_n=t\}$ is a partition of $[0,t]$, $|\Pi|$ is the mesh size of the partition, and the above limit is in probability.
 \end{defn}
\begin{defn}
 We say that   $\mathbb Q\in \mathcal{M}^{2,abs}_l$ if $\mathbb Q\in \mathcal{M}^2_l$ and $\langle X^i,X^j\rangle $ has absolutely continuous paths for each $i,j$. 
\end{defn}
 \begin{rmk}
 In the previous proposition, $\bigl\langle X\bigr \rangle_t=a(t)=\int_0^t G(s)ds$ under $\mathbb{Q}^a$. \end{rmk}

 \subsection{Gantert's inequality in $\mathbb R^l$}
 Now we proceed to extend \cite[Satz 1.3]{Ga91}, concerning a lower bound for the specific relative entropy. This is the content of Theorem \ref{thm:gantert_ineq} from the introduction, which we restate here, with more details, for the convenience of the reader:
 \begin{thm} \label{thm2.3}
     Let $B$ be the standard Brownian motion and $\mathbb{B}$ its law. Let $\mathbb{Q}\in \mathcal{M}^{2,abs}_l$ be such that $X_0=0$ a.s.\ and $\bigl\langle X\bigr \rangle_t = \int_0^t G(s,X)ds$, with $G $ a predictable symmetric positive semidefinite matrix-valued process. Then we have:
\begin{equation}\label{eq13}\begin{split}\textstyle 
h_l(\Q|\mathbb{B})\geq \frac{1}{2}\E_{\Q} \Big[ \int_0^1 tr(G(s,X)) -l -\log(\det G(s,X) ) ds\Big].
\end{split}\end{equation}
 \end{thm}
 \begin{proof}
 We proceed in parallel with the proof of \cite[Satz 1.3]{Ga91}.
 Let $n\in \N$. We start by defining a new measure on $(C, \mathcal{F}^n)$. Let $\mathbb{Q}^n$
be such that for  $k=1, \dots,n$:
\begin{equation*}\begin{split} \textstyle 
X_{\frac{k}{n}}-X_{\frac{k-1}{n}}|X_0,\dots,X_{\frac{k-1}{n}}&\sim_{\mathbb{Q}^n}\mathcal{N}_l\big( 0_l,\mathbb{E}_{\Q}[\Delta^n_kX (\Delta^n_kX)^T|X_0,\dots,X_{\frac{k-1}{n}}]\big ),
\end{split}\end{equation*}
i.e.\  the new measure is defined such that the discrete increments under $\mathbb Q^n$ have the same conditional covariances as under $\mathbb{Q}$.

Assume that $H(\mathbb{Q}|\mathbb{B})|_{\mathcal{F}^n} < \infty$. Then it follows that, $\mathbb{Q}|_{\mathcal{F}^n} \ll \mathbb{B}|_{\mathcal{F}^n}$, which implies, 
$\mathbb{Q}|_{\mathcal{F}^n} \ll \Q^n $ and $ \mathbb{Q}^n \ll \mathbb{B}|_{\mathcal{F}^n}$, and thus we have, 
 $\frac{d\mathbb{Q}}{d\mathbb{B}}|_{\mathcal{F}^n}=\frac{d\mathbb{Q}}{d\mathbb{Q}^n}|_{\mathcal{F}^n}\frac{d\mathbb{Q}^n}{d\mathbb{B}}|_{\mathcal{F}^n}$. 
Now taking the logarithm and integrating with respect to $\mathbb{Q}$ yields
\begin{equation}\label{eq14}\begin{split}
H(\Q|\mathbb{B})|_{\mathcal{F}^n} &= \textstyle \int_C \log \frac{d\Q}{d\Q^n}|_{\mathcal{F}^n} d\Q +\int_C \log \frac{d\Q^n}{d\mathbb{B}}|_{\mathcal{F}^n} d\Q \\ &\textstyle =H(\Q|\Q^n)|_{\mathcal{F}^n}+\int_C \log\frac{d\Q^n}{d\mathbb{B}}|_{\mathcal{F}^n}d\Q.
\end{split}\end{equation}

In the remainder we will show that
$\frac{1}{2^n}\int _C \log \frac{d\Q^{2n}}{d\mathbb{B}}|_{\mathcal{F}^{2^n}} d\Q$ 
converges to the right-hand side of \eqref{eq13}. 
The statement then follows, since $H(\Q|\Q^n)|_{\mathcal{F}^n} $ is non-negative. We assume wlog.\ $ H(\Q|B)|_{\mathcal{F}^{2^n}}< \infty$ for all $ n\in \N $ and thus \eqref{eq14} holds for all $ n \in \N$. 

Let $\mathcal{P}_n:= \sigma (A \times (s,t] \, | \, s<t \in \{ \frac{0}{n}, \frac{1}{n}, \dots , \frac{n}{n}\}, A \in \sigma(X_{\frac{0}{n}},X_{\frac{1}{n}}, \dots, X_s))$. Since $G$ is $\lambda \otimes  \Q$-integrable, we can define, 
\begin{align}
\label{eq:aprox_predict}
G_n:= \E _{\lambda \otimes  \Q}[G|\mathcal{P}_n].
\end{align}
We have for every $n \in \N$:
\begin{equation*}\begin{split}\textstyle 
G_n(t)=n \sum _{k=1}^n \E_{\Q}\Big[ \int^{\frac{k}{n}}_{\frac{k-1}{n}}G(s,X)ds| X_0, \dots, X_{\frac{k-1}{n}}\Big](X)\mathbbm{1}_{(\frac{k-1}{n},\frac{k}{n}]}(t).
\end{split}\end{equation*}
Note that this means,
\begin{equation*}\begin{split}\textstyle 
X_{\frac{k}{n}}-X_{\frac{k-1}{n}}|X_0,\dots,X_{\frac{k-1}{n}} \sim_{\Q^n} \mathcal{N}_l\Big( 0_l, \frac{1}{n}G_n \big(\frac{k}{n},X\big)\Big).
\end{split}\end{equation*}
Indeed, this follows easily from,
\begin{equation*}\begin{split}\textstyle 
\frac{1}{n} G_n\big(\frac{k}{n},X\big)
&=\textstyle  \E_{\Q}\Big[ \E_{\Q}\Big[ \int_{\frac{k-1}{n} }^{\frac{k}{n}}G(s,X)ds | \mathcal{F}_{\frac{k-1}{n}}\Big]|X_0,\dots,X_{\frac{k-1}{n}}\Big]\\
&=\textstyle \E_{\Q}\Big[ \E_{\Q}\Big[ \Delta^n_k X (\Delta^n_k X)^T | \mathcal{F}_{\frac{k-1}{n}}\Big]|X_0,\dots,X_{\frac{k-1}{n}}\Big]\\
&=\textstyle \E_{\Q}\big[\Delta^n_k X (\Delta^n_k X)^T | X_0,\dots,X_{\frac{k-1}{n}}\big] 
\end{split}\end{equation*}
and since we have assumed $\Q|_{\mathcal{F}^n}\ll \mathbb{B}|_{\mathcal{F}^n}$ we have $G(\frac{k}{n},X)>0$ $\Q$-a.s. We can now compute the Radon-Nikodym density of $\Q^n$ with respect to $\mathbb{B}|_{\mathcal{F}^n}$, 
\begin{equation*}\begin{split}\textstyle 
\frac{d\Q^n}{d\mathbb{B}}|_{\mathcal{F}^n}(X)&= \textstyle \prod _{k=1}^n \frac{d\mathcal{L}_{\Q^n}(\Delta_k^n X|X_0,\dots,X_{\frac{k-1}{n}})}{d\mathcal{L}_B(\Delta_k^n X|X_0,\dots,X_{\frac{k-1}{n}})}(\Delta_k^n X)\\
&=\textstyle \prod _{k=1}^n \frac{d \mathcal{N}_l(0_l, \frac{1}{n}G_n(\frac{k}{n},X))}{d\mathcal{N}_l(0_l, \frac{1}{n}\mathcal{I} )}(\Delta_k^n X)\\
&=\textstyle \prod _{k=1}^n  \sqrt{\frac{\det (\frac{1}{n}\mathcal{I})}{\det (\frac{1}{n}G_n)}} \, \frac{\exp (-\frac{1}{2}(\Delta_k^n X)^T (\frac{1}{n}G_n)^{-1}\Delta_k^n X)}{\exp (-\frac{1}{2}(\Delta_k^n X)^T (\frac{1}{n}\mathcal{I})^{-1}\Delta_k^n X)}\\
&=\textstyle \prod _{k=1}^n  \frac{1}{\sqrt{\det G_n}} \exp (-\frac{n}{2}(\Delta_k^n X)^T G^{-1}_n\Delta_k^n X) + \frac{n}{2}(\Delta_k^n X)^T \mathcal{I}^{-1}\Delta_k^n X)).
\end{split}\end{equation*}
Taking the logarithm, integrating with respect to $\Q$ and dividing by $n$, we get:
\begin{equation}\label{eq15}\begin{split}
&\textstyle \frac{1}{n}\int_C \log \frac{d\Q^n}{d\mathbb{B}}|_{\mathcal{F}^n} d\Q\\
=&\textstyle \frac{1}{n} \sum_{k=1}^n \E_{\Q} \Big[ \frac{n}{2} (\Delta_k^nX)^T( \mathcal{I})^{-1}\Delta_k^nX - \frac{n}{2} (\Delta_k^nX)^T( G_n)^{-1}\Delta_k^nX -\frac{1}{2} \log (\det G_n)\Big]\\
=&\textstyle \frac{1}{n} \sum_{k=1}^n \E_{\Q} \Big[ \frac{1}{2} tr(n\mathcal{I}^{-1}\Delta_k^nX(\Delta_k^nX)^T )-\frac{1}{2} tr(nG^{-1}_n\Delta_k^nX(\Delta_k^nX)^T )-\frac{1}{2} \log (\det G_n)\Big].
\end{split}\end{equation}
By the tower property we have \begin{equation*}\begin{split}
&\textstyle \E_{\Q}\Big[ n G_n\big(\frac{k}{n}, X\big)^{-1}\Delta_k^nX(\Delta_k^nX)^T\Big] \\
=&\textstyle \E_{\Q}\Big[ n G_n\big(\frac{k}{n}, X\big)^{-1} \E_{\Q}\Big[\Delta_k^nX(\Delta_k^nX)^T|X_0, \dots, X_{\frac{k-1}{n}} \Big]\Big]
=\mathcal{I}.
\end{split}\end{equation*}
Continuing in \eqref{eq15}, we find 
\begin{equation*}\begin{split}\textstyle 
\frac{1}{n}\int_C \log \frac{d\Q^n}{d\mathbb{B}}|_{\mathcal{F}^n} d\Q&\textstyle = \frac{1}{n} \sum_{k=1}^n \E_\Q \left [ \frac{1}{2}tr\Big(G_n\big(\frac{k}{n},X\big) \Big) -\frac{l}{2}-\frac{1}{2} \log \Big(\det G_n\big(\frac{k}{n},X\big) \Big)\right ]\\
&\textstyle =\frac{1}{n} \sum_{k=1}^n \E_{\Q}\left[ F_l\Big(G_n\big(\frac{k}{n},X\big)\Big) \right]\\
&\textstyle =\E_{\Q}\left[\int_0^1 F_l\Big(G_n\big(s,X\big)\Big) ds \right],
\end{split}\end{equation*}
where we have used that $t \rightarrow G_n(t,X)$ is constant on intervals of the form $(\frac{k-1}{n}, \frac{k}{n}]$. 
We now use that $F_l$ is convex and apply Jensen’s inequality to obtain the bound
\begin{equation*}\begin{split}\textstyle 
E_{\Q}\left[\int_0^1 F_l\left(G_n\left(s,X\right)\right)ds \right]&=\textstyle  E_{\Q}\left[\int_0^1 F_l(\E_{\lambda \otimes \Q}[G|\mathcal{P}_n] (s,X))ds \right]\\& \textstyle \leq E_{\Q}\left[\int_0^1\E_{\lambda \otimes \Q}[F_l(G)|\mathcal{P}_n] (s,X)ds \right] \\ &= \textstyle E_{\Q}\left[\int_0^1 F_l(G(s,X))ds \right], 
\end{split}\end{equation*}
where the last equality is due to Fubini’s theorem and the tower property.

Now we restrict ourselves to the sub-sequence of powers of 2. Note that $(\mathcal{P}_{2^n})_{n\geq 1}$ is a sequence of sigma algebras that increases to the predictable sigma. Therefore, $G_{2^n}$  is a closed martingale with respect to $(\mathcal{P}_{2^n})_{n\in \N}$  and from the convergence theorem for uniformly integrable martingales, it follows that $\lambda \otimes \Q$-a.s., 
\begin{equation*}\begin{split}
G_{2^n}(t,X) \xrightarrow{n\rightarrow \infty} G(t,X). 
\end{split}\end{equation*}
With Fatou’s Lemma and $F \geq 0$ and the upper bound from before we get, 
\begin{equation*}\begin{split}\textstyle 
E_{\Q}\left[\int_0^1 F_l(G(s,X))ds \right]&=\textstyle 
E_{\Q}\left[\int_0^1 \liminf_{n\rightarrow \infty}F_l(G_{2^n}(s,X))ds \right]\\&\textstyle \leq  \liminf_{n\rightarrow \infty}E_{\Q}\left[\int_0^1 F_l(G_{2^n}(s,X))ds \right]\\
&\textstyle \leq \limsup_{n\rightarrow \infty}E_{\Q}\left[\int_0^1 F_l(G_{2^n}(s,X))ds \right]\\& \textstyle \leq E_{\Q}\left[\int_0^1 F_l(G(s,X))ds \right].
\end{split}\end{equation*}
We conclude, 
\begin{equation*}\begin{split}\textstyle 
\lim_{n\rightarrow \infty} \frac{1}{2^n} \int_C \log \frac{d\Q^{2^n}}{d\mathbb{B}} |_{\mathcal{F}^{2^n}}d\Q= \E_{\Q}\left[\int_0^1 F_l(G(s,X)) ds\right]. 
\end{split}\end{equation*}

 \end{proof}
 \begin{rmk}
 Note that all arguments in the proof of Theorem \ref{thm2.3} 
still hold, if the sequence  $(2^n)_{n \in \N}$ is replaced by $(p^n)_{n\in \N}$ for any $p \in \N$ and $p\geq 2$.
\end{rmk}

In dimension 1 (i.e.\ for $l=1$), beyond special examples, it was first established in \cite{BaUn22} that the inequality in Theorem \ref{thm2.3} is in fact an equality in case that $\mathbb Q$ is the law of a martingale diffusion SDE with a very regular diffusion coefficient. In the present paper we do not aim to obtain a generalization of that result. However, the closed-form formulae that we will obtain provide positive evidence that a similar result is to be expected also in higher dimensions.

\section{On the convex lower semicontinuous envelope}

The aim of this part is to prove Theorem \ref{thm:hull}. We define
\[\textstyle L_l(\mathbb Q): = \mathbb{E}_{\mathbb Q}\left[\int_0^1 F_l(\Sigma_t) dt\right ],\]
in case $\mathbb Q\in \mathcal M_l^{2,abs}$ and under $\mathbb Q$ we have $\dot{\langle X \rangle_t}=\Sigma_t$; otherwise we set  $L_l(\mathbb Q): =+\infty$.

In the next two lemmas we collect important information about this functional.

\begin{lem}\label{lem:topo_prop}
    $L_l$ is convex and weakly lower semicontinuous.
\end{lem}

\begin{proof}
Recall that  $F_l(\Sigma):= \frac{1}{2}(tr(\Sigma)-l-\log(\det \Sigma))$. We can extend $L_l$ to the set of continuous semimartingale laws, by letting $\tilde F_l(b,\Sigma):=+\infty$ if $b\in\mathbb R^l\backslash \{0\}$ and otherwise $\tilde F_l(0,\Sigma):= F_l(\Sigma)$. The extension is then denoted $\tilde L_l(\mathbb Q): = \mathbb{E}_{\mathbb Q}\left[\int_0^1 \tilde F_l(b,\Sigma_t) dt\right ]$, with the understanding that now the canonical process is allowed to have a drift $\int_0^\cdot b_tdt$ under $\mathbb Q$. We also introduce 
 for $\epsilon>0$ the function    $\tilde F_l^\epsilon(b,\Sigma):= \frac{1}{2}(tr(\Sigma)-l-\log(\det( \Sigma+\epsilon\mathcal I)))$, so $\tilde F_l^\epsilon\nearrow \tilde F_l$ pointwise. Furthermore   $\tilde F_l^\epsilon$ is jointly convex, since $\log\det $ is concave (see e.g.\ the proof of Lemma \ref{F_l}). We can then verify all the assumptions behind \cite[Theorem 8.3]{BaPa22}, showing that $\tilde L_l^\epsilon$ is lower semicontinuous. Hence, by taking supremum in $\epsilon>0$, also $\tilde L_l$ is lower semicontinuous. This in turn implies that $L_l$ is lower semicontinuous. The convexity of $L_l$ is immediate from the concavity of $\log\det $.
\end{proof}

We recall that $\lambda$ denotes the Lebesgue measure on $[0,1]$.

\begin{lem}\label{lem:Karandakar}
Suppose $M$ is an $\mathbb R^l$-valued martingale defined on some filtered probability space $(\tilde \Omega,\tilde{\mathcal F},(\tilde{\mathcal F}_t)_t,\tilde{\mathbb P})$, such that $\mathcal L(M) \in \mathcal M_l^{2,abs}$. Suppose further that $dM_t=\sigma_tdW_t$   for some $l$-dimensional  $(\tilde{\mathcal F}_t)_t$-Brownian motion $W$ and matrix-valued adapted process $(\sigma_t)_t$. Then
\[\textstyle L_l(\mathcal L(M) ) = \tilde{\mathbb E}\left[\int_0^1 F_l(\sigma_t\sigma_t^T)dt \right].\]
\end{lem}

\begin{proof}
Recall that it is possible to construct the quadratic variation of a continuous martingale in pathwise fashion, and by polarization the quadratic covariation can also be constructed this way. Hence this applies to continuous martingales taking values in $\mathbb R^l$. The construction therefore does not depend on the filtration or the reference measue. See Karandikar's \cite[Theorem and Remark 5]{Ka83}. As a consequence, if $X$ denotes the canonical process of continuous $\mathbb R^l$-valued paths, then  
\begin{align}\label{eq:formula_dens_qv}
\hat \Sigma_s:=\liminf_{n\to\infty} n\{\langle X \rangle_s - \langle X \rangle_{(s-1/n)_+}  \},
\end{align}
is well-defined $\mathbb Q$-a.s.\ simultaneously for all $\mathbb Q$ continuous martingale measure. If now $\mathbb Q\in \mathcal M_l^{2,abs}$, then $\mathbb Q$-a.s.\ for all $t\in[0,1]$:
\[\textstyle \langle X \rangle_t= \int_0^t \hat \Sigma_s ds, \]
i.e.\ $\hat \Sigma$ is a version (in the $\lambda \times \mathbb Q$-sense) of the density of the quadratic covariation of $X$. If $\mathbb Q:=\text{Law}(M)$ we conclude that 
$\sigma\sigma^T = \hat \Sigma \circ M $ 
in the $\lambda\times \tilde{\mathbb P}$-sense. Hence
\[\textstyle \tilde{\mathbb E}\left[\int_0^1 F_l(\sigma_t\sigma_t^T)dt \right] = \mathbb{E}_{\mathcal L(M)} \left[\int_0^1  F_l(\hat \Sigma_t) dt \right]= L_l(\mathcal L(M)). \]
\end{proof}

\begin{rmk}\label{rem:Karandakar}
    The previous proof establishes that the functional $L_l$ is affine. Indeed, if $\Sigma^{\mathbb Q}$ is (a version of) the density of the quadratic covariation under $\mathbb Q$ then the proof shows that $\mathbb Q(\lambda\{t\in[0,1]:\Sigma_t^{\mathbb Q}\neq\hat \Sigma_t\})=0$, with $\hat \Sigma$ as in \eqref{eq:formula_dens_qv}.
\end{rmk}

With the aim of proving Theorem \ref{thm:hull}, the following auxiliary  class of martingales proves very useful. Let $N\in\mathbb N$ be fixed. We say that a martingale measure $\mathbb Q $ is in $\mathcal M_l^2(N)$, if it is the law of a martingale $M$ obtained as follows: for some measurable process $(\sigma_t)_t$, square-integrable, with values on the set of symmetric positive definite matrices of size $l\times\l$, such that
$$\sigma_t \text{ is $\mathcal F_{(t-\frac{1}{N})\vee 0}$-measurable},$$
we have $M_t=M_0+\int_0^t\sigma_sdB_s$. The key observation is that then
$$\textstyle \text{Law}(M_{k/N}|\mathcal F_{(k-1)/N})=\mathcal N\left(M_{(k-1)/N},\int^{k/N}_{(k-1)/N}\sigma^2_udu\right),$$
as $\sigma_u$ is $\mathcal F_{(k-1)/N}$-measurable for $u\leq k/N$.

\begin{lem}\label{lem:equ_case_delay}
  Let $\mathbb Q \in \mathcal M_l^2(N)$ with $\mathbb Q = \text{Law}(M)$ as above. Then
        \begin{equation}\label{eq:Gantert_h_eq_N}\textstyle 
     h_{l}(\mathbb{Q}|\mathbb{B}^l)= \mathbb{E}\left[\int_0^1 F_l(\sigma_t^2) dt\right ].
 \end{equation}
\end{lem}

\begin{proof}
    By Theorem \ref{thm2.3} we only need to show the inequality ``$\leq $''. For notational simplicity we assume that $N=2^{n_0}$, but this is not important. We now take $n>n_0$. We will show that 
            \begin{equation}\label{eq:Gantert_h_eq_N_interim}\textstyle 
    2^{-n} H(\mathbb Q |\mathbb{B}^l)|_{\mathcal{F}^{2^n}}\leq \mathbb{E}\left[\int_0^1 F_l(\sigma_t^2) dt\right ],
 \end{equation}
 so the result follows by taking $\liminf_n$. By the additivity of the relative entropy (e.g.\ Lemma \ref{lem1.3}), 
 \[\textstyle H(\mathbb Q |\mathbb{B}^l)|_{\mathcal{F}^{2^n}} =\mathbb E\left[\sum_{k=0}^{2^{n}-1} A_k(M_{02^{-n}},\dots, M_{k2^{-n}})\right],\] 
 where
 \[A_k(x_0,\dots,x_k):= H(\mathcal L(M_{(k+1)2^{-n}}\,|\, M_{02^{-n}}=x_0,\dots,M_{k2^{-n}}=x_k) \, | \, \mathcal N( x_k, 2^{-n}\mathcal I ) ) ,\]
 where we applied the independent increments property of Brownian motion. By convexity of the relative entropy, the tower property of conditional expectations, and Jensen's inequality, we have the bound
 \begin{align*}
 &\mathbb E[ A_k(M_{02^{-n}},\dots, M_{k2^{-n}})] \\
 =& \mathbb E \big[ H(\mathbb E[\mathcal L(M_{(k+1)2^{-n}}\,| \mathcal F_{k2^{-n}}) \, |\, M_{02^{-n}},\dots,M_{k2^{-n}}] \, | \, \mathcal N( M_{k2^{-n}}, 2^{-n}\mathcal I ) )\big] \\
 \leq & \mathbb E \big[\mathbb E[H(\mathcal L(M_{(k+1)2^{-n}}\,| \mathcal F_{k2^{-n}}) \, |\, \mathcal N( M_{k2^{-n}}, 2^{-n}\mathcal I ))\, |\, M_{02^{-n}},\dots,M_{k2^{-n}}] \big] \\
 =& \mathbb E \big[H(\mathcal L(M_{(k+1)2^{-n}}\,| \mathcal F_{k2^{-n}}) \, |\, \mathcal N( M_{k2^{-n}}, 2^{-n}\mathcal I )) \big] .
 \end{align*}
 By construction $\mathcal L (M_{(k+1)2^{-n}}|\mathcal F_{k2^{-n}})=\mathcal N (M_{k2^{-n}},\int_{k2^{-n}}^{(k+1)2^{-n}}\sigma^2_udu )$, since $\sigma_u$ is $ \mathcal F_{k2^{-n}}$-measurable if $u\leq (k+1)2^{-n}$, as $(k+1)2^{-n}-2^{-n_0}\leq k2^{-n}$. Hence
 \begin{align*}
     & \mathbb E[ A_k(M_{02^{-n}},\dots, M_{k2^{-n}})] \\
      \leq & \textstyle  \mathbb E \left[H\left(\mathcal N\left(M_{k2^{-n}}, \int_{(k-1)2^{-n}}^{k2^{-n}}\sigma_u^2du \right ) \, |\, \mathcal N( M_{k2^{-n}}, 2^{-n}\mathcal I )\right ) \right] \\
      =&\textstyle  \mathbb E\left[ F_l\left( 2^n \int_{(k-1)2^{-n}}^{k2^{-n}}\sigma_u^2du  \right) \right],
 \end{align*}
 where for the last step we used \eqref{eq:prelim_ent_Gaussians}. By the convexity of $F_l$ (per Lemma \ref{F_l}) and Jensen's inequality, we find
 \[ \textstyle    \mathbb E[ A_k(M_{02^{-n}},\dots, M_{k2^{-n}})] \leq \mathbb E\left[ 2^n\int_{(k-1)2^{-n}}^{k2^{-n}} F_l\left(  \sigma_u^2  \right) du\right] ,\]
 so finally we obtained the desired $H(\mathbb Q |\mathbb{B}^l)|_{\mathcal{F}^{2^n}} \leq 2^n \mathbb E[\int_0^1  F_l(\sigma_u^2)du] $.
\end{proof}

The next result is technical but key for us:

\begin{lem}\label{lem:many_steps}
    For every $f$ continuous bounded function on $C$ (we write $f\in C_b(C)$), we have
    $$\textstyle \sup_{\mathbb Q\in\mathcal M_l^2}\left\{ \mathbb E_{\mathbb Q}[f(X)] - L_l(\mathbb Q) \right\} =\sup_{\substack{n\in\mathbb N , \mathbb Q\in\mathcal M_l^2(n)}}\left\{ \mathbb E_{\mathbb Q}[f(X)] - L_l(\mathbb Q) \right\}.$$
\end{lem}

\begin{proof}
It suffices to show: for all $\mathbb Q\in \mathcal M_l^{2,abs}$ martingale measure with $L_l(\mathbb Q)<\infty$, there is $\mathbb Q_n\in\mathcal M_l^2(n) $ with $\mathbb Q_n\to \mathbb Q$ and $L_l(\mathbb Q) \geq \liminf_n L_l(\mathbb Q_n)$. Hence call $\Sigma$ the density of the quadratic covariation of $X$ under $\mathbb Q$     and $\sigma_t:=\Sigma_t^{1/2}$ the unique symmetric  positive definite square root. As $L_l(\mathbb Q)<\infty$, then $\sigma$ is a.s.\ invertible, and so $dX_t= \sigma_t dB_t $ for the Brownian motion $B_t:=\int_0^t \sigma_s^{-1}dX_s $. \\

\emph{Step 1:} For $\epsilon>0$ define $Z_t=X_t+\epsilon B_t$, so $dZ_t = (\sigma_t+\epsilon\mathcal I)dB_t$. By Lemma \ref{lem:Karandakar}:
\begin{align*}  
L_l(\mathcal L(Z)) = &\textstyle \mathbb E_{\mathbb Q} \left [ \int_0^1 F_l(  (\sigma_t+\epsilon\mathcal I)^2)dt  \right ] \\ \leq & \textstyle \frac{1}{2} \mathbb E_{\mathbb Q} \left [ \int_0^1 \{tr( (\sigma_t+\epsilon\mathcal I)^2) -l - \log\det  (\sigma_t^2)  \}dt  \right ]   \to  L_l(\mathbb Q),
\end{align*}
    where in the inequality we used that $\log\det$ is increasing for the positive semidefinite matrix ordering. Sending $\epsilon\to 0$ we see that the law of $Z$ can be made arbitrarily close to $\mathbb Q$. In conclusion, we may from now on assume that wlog.\ $\Sigma_t\geq \bar\epsilon \mathcal I$ under $\mathbb Q$, in the sense of positive semidefinite matrices.\\

    \emph{Step 2:} For $n\in\mathbb N$, let $\Sigma^{(n)}(t,\omega)$ be the matrix with eigenspaces equal to those of $\Sigma(t,\omega)$ and whose eigenvalues are of the form $\alpha\wedge n$ whenever $\alpha$ is a corresponding eigenvalue of $\Sigma(t,\omega)$. We denote by $\sigma^{(n)}(t,\omega)$ the unique symmetric positive semidefinite square root of $\Sigma^{(n)}(t,\omega)$. We define $Z^{(n)}$ via $Z^{(n)}_0:=X_0$ and $dZ^{(n)}_t:= \sigma^{(n)}(t,X)dB_t$. Thus by Lemma \ref{lem:Karandakar}:
 \begin{align*}  
L_l(\mathcal L(Z^{(n)})) = &\textstyle  \mathbb E_{\mathbb Q} \left [ \int_0^1 F_t(\Sigma^{n}(t,X))  dt\right ]
\to  L_l(\mathbb Q),
\end{align*}
    where the limit is due to monotone convergence. Hence from now on we may further assume that $\Sigma_t\leq (\bar\epsilon)^{-1} \mathcal I$ under $\mathbb Q$, in the sense of positive semidefinite matrices.\\

     \emph{Step 3:} Define $\sigma^{(n)}_t=n\int^t_{(t-1/n)\vee 0}\sigma_s dt$, so $\sigma^{n}\to\sigma$ in the $\lambda\times \mathbb Q$-sense. Remark that the process $(\sigma^{(n)}_t)_t$ is adapted and has continuous paths.
      Thanks to Steps 1-2 we may apply dominated convergence so that
     \[\textstyle \mathbb E_{\mathbb Q}\left[\int_0^1 F_l((\sigma^{(n)}_t)^2)dt\right]\to \mathbb E_{\mathbb Q}\left[\int_0^1 F_l(\sigma_t^2)dt\right].\]
     Similarly, if $Z^{(n)}_t:=X_0+\int_0^t \sigma^{(n)}_s dB_s$, then the BDG inequality shows that
     \[\textstyle \mathbb E_{\mathbb Q}\left[\sup_{t\in[0,1]}|Z^{(n)}_t - X_t|^2 \right]\to 0. \]
     Hence, we may further assume from now on that $(\sigma_t)_t$ has continuous paths.\\

    \emph{Step 4:} We extend $\sigma_t$ to $t<0$ by setting  it to be equal to the identity matrix. This induces a discontinuity at $t=0$ which plays no role in the coming arguments. Define
    $Z^{(n)}_t:= X_0 + \int_0^t\sigma_{s-1/n}dB_s$ and call $\sigma^{(n)}_t:= \sigma_{t-1/n}$, so $\sigma^{(n)}_t$ is $\mathcal F_{(t-1/n)\vee 0}$-measurable. By continuity of paths $\sigma^{(n)}_t\to\sigma_t$ for each $t>0$. By dominated convergence, as in Step 3, we conclude  
    \[\textstyle \mathbb E_{\mathbb Q}\left[\int_0^1 F_l((\sigma^{(n)}_t)^2)dt\right]\to \mathbb E_{\mathbb Q}\left[\int_0^1 F_l(\sigma_t^2)dt\right],\]
    as well as 
      \[\textstyle \mathbb E_{\mathbb Q}\left[\sup_{t\in[0,1]}|Z^{(n)}_t - X_t|^2 \right]\to 0. \]
By definition $\mathbb Q_n:=\mathcal L(Z^{(n)})$ is in   $\mathcal M_l^2(n)$, so we conclude.
\end{proof}

\begin{rmk}
    The above proof works the same if we replace the functional $\mathbb Q\mapsto \mathbb E_{\mathbb Q}[f(X)]$ by the possibly non-linear one $\mathbb Q\mapsto G(\mathbb Q)$, if the latter is continuous. Notice also that at every step in the proof we approximated $\mathbb Q$ by a sequence of better behaved approximates, and the convergence of these to $\mathbb Q$ was not only in the weak convergence sense, but also included second moment convergence. Hence the functional $\mathbb Q\mapsto G(\mathbb Q)$ actually only needs to be continuous wrt.\ a 2-Wasserstein topology. We may pass to lower-semicontinuous $G$ by standard arguments.
\end{rmk}

We can finally provide the proof of Theorem \ref{thm:hull}:

\begin{proof}[Proof of Theorem \ref{thm:hull}]
For simplicity we write $h_l(\mathbb Q)$ instead of $h_l(\mathbb Q)|\mathbb{B}^l)$.    We define $h_l^*(f)= \sup_{\mathbb Q\in\mathcal M_l^2}\{\mathbb E_{\mathbb Q}[f(X)] - h_l(\mathbb Q)\}$ and $h_l^{**}(\mathbb Q)= \sup_{f\in C_b(C)}\{\mathbb E_{\mathbb Q}[f(X)] - h_l^*(f)\} $. We know that $h_l^{**}$ is the largest convex lower semicontinuous minorant of $h_l$. By Lemmas \ref{lem:many_steps} and \ref{lem:equ_case_delay} we have:
    \begin{align*}
    h_l^*(f)&\textstyle \geq \sup_{\substack{n\in\mathbb N , \mathbb Q\in\mathcal M_l^2(n)}}\left\{ \mathbb E_{\mathbb Q}[f(X)] - h_l(\mathbb Q) \right\} \\
    &=\textstyle  \sup_{\substack{n\in\mathbb N , \mathbb Q\in\mathcal M_l^2(n)}}\left\{ \mathbb E_{\mathbb Q}[f(X)] - L_l(\mathbb Q) \right\} \\
    &=\textstyle  \sup_{\mathbb Q\in\mathcal M_l^2}\left\{ \mathbb E_{\mathbb Q}[f(X)] - L_l(\mathbb Q)\right\} =: L_l^*(f).    
    \end{align*}
    It clearly follows that $h_l^{**}(\mathbb Q)\leq L_l^{**}(\mathbb Q)$. Since by Lemma \ref{lem:topo_prop} $L_l$ is convex, proper, and lower semicontinuos, we have $L_l(\mathbb Q)= L_l^{**}(\mathbb Q)$. By Theorem \ref{thm:gantert_ineq} we know that $L_l$ is a minorant of $h_l$. We conclude that $L_l= h_l^{**}$.
\end{proof}

\section{Tensorization and closed-form formulae}

We first establish some useful tensorization results wherein the specific relative entropy in multiple dimensions is compared to its one-dimensional counterpart. Next we proceed to compute the closed-form expression for the specific relative entropy between simple multivariate martingales as promised in the introduction.

\subsection{Tensorization properties}
We establish two tensorization (in)equalities, the second of which was already mentioned in the introduction:
\begin{lem} \label{lem3.1}
    Given continuous martingales  $M:=( M^1 ,\dots, M^l   ),\,N:=( N^1 ,\dots ,  N^l  )$  such that $\mathcal L(M^1)=\mathcal L(M^i)$, $\mathcal L(N^1)=\mathcal L(N^i)$, and all the $M^i$'s are independent of each other and similarly all the $N^i$'s are independent of each other, then
    \begin{equation*} h_l(\mathcal{L}(M)|\mathcal{L}(N))=l h_1(\mathcal{L}(M^1)|\mathcal{L}(N^1)).\end{equation*}
\end{lem}
\begin{proof}
It is well-known, or it can be readily derived from Lemma \ref{lem1.3}, that if $Z^1,\dots,Z^l$ and $Y^1,\dots,Y^l$ are random vectors, with $Y^1,\dots,Y^l$ independent of each other and similarly $Z^1,\dots,Z^l$ independent of each other, then  
\[\textstyle H(\mathcal L(Z^1,\dots, Z^l) | \mathcal L(Y^1,\dots, Y^l)) = \sum_{i=1}^l H(\mathcal L (Z^i)|\mathcal L (Y^i)).\]
Taking $Z^i:= (M_{k2^{-n}}^i)_{k=0}^{2^n}$ and $Y^i:= (N_{k2^{-n}}^i)_{k=0}^{2^n}$, and observing that therefore $\mathcal L(Z^1)=\mathcal L(Z^i)$ and $\mathcal L(Y^1)=\mathcal L(Y^i)$, we conclude that
\begin{align*}
   & H(\mathcal L(M_{02^{-n}},\dots, M_{2^n2^{-n}}) | \mathcal L(N_{02^{-n}},\dots, N_{2^n2^{-n}}))\\=& lH(\mathcal L(M^1_{02^{-n}},\dots, M^1_{2^n2^{-n}}) | \mathcal L(N^1_{02^{-n}},\dots, N^1_{2^n2^{-n}}))  .
\end{align*}
dividing by $2^n$ and taking liminf, we conclude.
\end{proof}

\begin{lem}\label{lem:tens_lb}
    Given two multivariate continuous martingales $M= ( M^1, \dots ,M^l  )$ and $N=( N^1 , \dots , N^l )$, where  $N$ has  independent coordinates, then the specific  relative entropy between $M$ and $N$ is at least equal to the sum of the  specific relative entropy between the coordinate processes:
    \begin{equation*} \textstyle h(\mathcal{L}(M)|\mathcal{L}(N))\geq \sum _{i=1}^l h(\mathcal{L}(M^i)|\mathcal{L}(N^i)).\end{equation*}
\end{lem}
\begin{proof}
It easily follows from the following fact: if $Z^1,\dots,Z^l$ and $Y^1,\dots,Y^l$ are random vectors, with $Y^1,\dots,Y^l$ independent of each other, then
\[\textstyle H(\mathcal{L}(Z^1,\dots,Z^l) | \mathcal{L}(Y^1,\dots,Y^l)  ) \geq \sum_{i=1}^l H(\mathcal{L}(Z^i)|\mathcal{L}(Y^i)). \]
Indeed, we can take $Z^i:= (M_{k2^{-n}}^i)_{k=0}^{2^n}$ and $Y^i:= (N_{k2^{-n}}^i)_{k=0}^{2^n}$ and apply the definition of the specific relative entropy. For the convenience of the reader we provide the argument behind the aforementioned fact: By Lemma \ref{lem1.3}, we have 
\begin{align*}&\textstyle  H(\mathcal{L}(Z^1,\dots,Z^l) | \mathcal{L}(Y^1,\dots,Y^l)  )=  \sum_{k=1}^l \mathbb E[H(\mathcal{L}(Z^k|Z^1=z^1,\dots,Z^{k-1}=z^{k-1})|\\&|\mathcal{L}(Y^k|Y^1=z^1,\dots,Y^{k-1}=z^{k-1}))_{z^1=Z^1,\dots,z^{k-1}=Z^{k-1}}] \\
=& \textstyle \sum_{k=1}^l \mathbb E[H(\mathcal{L}(Z^k|Z^1=z^1,\dots,Z^{k-1}=z^{k-1})|\mathcal{L}(Y^k))_{z^1=Z^1,\dots,z^{k-1}=Z^{k-1}}] \\
\geq&\textstyle  \sum_{k=1}^l H(\mathbb E[\mathcal{L}(Z^k|Z^1,\dots,Z^{k-1}])|\mathcal{L}(Y^k)) \\
=& \textstyle \sum_{k=1}^l H(\mathcal{L}(Z^k)|\mathcal{L}(Y^k),
\end{align*}
where the inequality is due to Jensen and the fact that the relative entropy is convex in its arguments, and the last equality is due to the power property.
\end{proof}

We now proceed to obtain closed form formulae for the specific relative entropy.
{We first recall that  an $l$-dimensional martingale Black-Scholes model fulfills $dM_t^{\Gamma}:=diag(M_t^{\Gamma})\Gamma dB_t$, namely $dM_t^i:=M_t^i(\sum_{k=1}^l\Gamma_{ik} d B_t^k)$ for $i=1,\dots,l$.}

\subsection{{Specific Relative Entropy between a multivariate Black-Scholes model and a Brownian motion}}\label{subsub2}
 
{
We provide the proof of Lemma \ref{lem:Black-Scholes_vs_BM_intro} from the introduction, which we rephrase here for the reader's convenience.}

\begin{lem}\label{lem:Black-Scholes_vs_BM}
{The specific relative entropy between a $\R^l$-valued martingale Black-Scholes models with parameter $\Gamma\in \mathbb{R}^{l \times l}$ and Brownian motion, both starting wlog.\ at $\mathbf{1}$, is given by:
\begin{equation} \begin{split} \label{srE_GBM_BM_l} \textstyle 
h_l(\mathcal L(M^{\Gamma})|\mathbb B^l)=\frac{1}{2}\Big(\sum_{i=1}^l \big(e^{\sum_{k=1}^l \Gamma_{ik}^2}-1\big)-l+\frac{1}{2} \sum_{i=1}^l \sum _{k=1}^l \Gamma_{ik}^2 - \log \big(\det (\Gamma \Gamma^T) \big)\Big).
\end{split}
\end{equation}}

\end{lem}

\begin{proof}
   {Let $B_t^l \sim \mathcal{N}_l(\mathbf{1}, t \mathcal{I})$ be the $\R^l$-valued Brownian motion with $B_0=\mathbf{1}$. Let $M_t^{\Gamma} \sim lognormal(\mu_t, \Sigma_{t})$ be the $\R^l$-valued martingale Black-Scholes model with $(\mu_t)_i:=-\frac{t}{2} \sum_{k=1}^l \Gamma_{ik}^2$ and $(\Sigma_{t})_{ij}:=t(\Gamma \Gamma^T)_{ij}=t \sum_{k=1}^l \Gamma_{ik}\Gamma{jk}$ for $i,j=1,\dots,l$.
    We are first interested in computing: 
    \begin{equation*} \begin{split} \textstyle 
H(\mathcal L(M^{\Gamma})|\mathbb B^l)|_{\mathcal{F}^n}=H\big(\mathcal{L}(M_0^{\Gamma},M_{\frac{1}{n}}^{\Gamma} , \dots,M_{\frac{n-1}{n}}^{\Gamma},M_1 ^{\Gamma})\,|\, \mathcal{L}(B_0^l, B_{\frac{1}{n}}^l, \dots , B_{\frac{n-1}{n}}^l,B_1^l ) \big),
\end{split}
\end{equation*}
which, expanding the coordinates and omitting the superscripts "$\Gamma, l$" from now on, becomes:
\begin{equation*} \begin{split} \textstyle 
=H\big( \mathcal{L}(M_0^1, \dots, M_0^l, M_{\frac{1}{n}}^1, \dots, M_{\frac{1}{n}}^l, \dots  )|\mathcal{L}(B_0^1, \dots, B_0^l, B_{\frac{1}{n}}^1, \dots, B_{\frac{1}{n}}^l, \dots  )\big).
\end{split}
\end{equation*}
Now, by Lemma \ref{lem1.3}, this becomes:
\begin{equation*} \begin{split} 
&=\textstyle H(\mathcal{L}(M_0^1, \dots, M_0^l)|\mathcal{L}(B_0^1, \dots, B_0^l)) +H(\mathcal{L}(M_{\frac{1}{n}}^1, \dots, M_{\frac{1}{n}}^l)|\mathcal{L}(B_{\frac{1}{n}}^1, \dots, B_{\frac{1}{n}}^l))+\\&\textstyle \int_{y_1, \dots, y_l} H\Big(\mathcal{L}\big(M_{\frac{2}{n}}^1, \dots, M_{\frac{2}{n}}^l|M_{\frac{1}{n}}^1=y_1, \dots, M_{\frac{1}{n}}^l=y_l\big) |\\&| \mathcal{L}\big(B_{\frac{2}{n}}^1, \dots, B_{\frac{2}{n}}^l|B_{\frac{1}{n}}^1=y_1, \dots, B_{\frac{1}{n}}^l=y_l\Big)
+\int _{y_1, \dots, y_l, z_1, \dots, z_l} H\Big(\mathcal{L}\big(M_{\frac{3}{n}}^1, \dots, M_{\frac{3}{n}}^l|\\&\textstyle | M_{\frac{1}{n}}^1=y_1, \dots, M_{\frac{1}{n}}^l=y_l, M_{\frac{2}{n}}^1=y_1, \dots, M_{\frac{2}{n}}^l=y_l\big)|\mathcal{L}\big(B_{\frac{3}{n}}^1, \dots, B_{\frac{3}{n}}^l|\\&\textstyle | B_{\frac{1}{n}}^1=y_1, \dots, B_{\frac{1}{n}}^l=y_l, B_{\frac{2}{n}}^1=y_1, \dots, B_{\frac{2}{n}}^l=y_l\big) \Big).
\end{split}
\end{equation*}
By the Markov property of geometric Brownian motions, in the conditioning we can just consider the $M_t^1, \dots, M_t^l$ with the greatest time $t$. Denoting  with $M_t$ and $B_t$ the coordinates $M_t^1, \dots, M_t^l$ and $B_t^1, \dots, B_t^l$ respectively, we want to compute: 
\begin{equation*}\begin{split} \textstyle 
H(\mathcal{L}(M_{\frac{1}{n}})|\mathcal{L}(B_{\frac{1}{n}}))+\sum_{m=2}^{n} \int_y H(\mathcal{L}(M_{\frac{m}{n}}|M_{\frac{m-1}{n}}=y)|\mathcal{L}(B_{\frac{m}{n}}|B_{\frac{m-1}{n}}=y))d\mathcal{L}(M_{\frac{m-1}{n}})(y).
\end{split} \end{equation*}
The first term is given by Remark \ref{ex_extra} with $t=\frac{1}{n}$:
\begin{align}  \label{ex_extra_1/n}
& \textstyle H(\mathcal{L}(M_{1/n})|\mathcal{L}(B_{1/n})) \\= &\textstyle \frac{1}{2}\Big( -\log \big( \det (\Gamma \Gamma^T)\big) + \frac{1}{n} \sum_{i=1}^l \sum_{k=1}^l \Gamma_{ik}^2 -l +n  \sum_{i=1}^l \big(e^{\frac{1}{n}\sum
_{k=1}^l\Gamma_{ik}^2}-1 \big) \Big). \notag
\end{align}
For the rest, we use that $\mathcal{L}(B_{\frac{m}{n}}|B_{\frac{m-1}{n}}=y)\sim \mathcal{N}(y, \frac{1}{n}\mathcal{I})$ and  $\mathcal{L}(M_{\frac{m}{n}}|M_{\frac{m-1}{n}}=y) \sim lognormal(\mu_{\frac{1}{n}}+\log y, \Sigma_{\frac{1}{n}})$ where $y \in \R^l$, $\mu_{\frac{1}{n}}=-\frac{1}{2n}\sum _{k=1}^l\Gamma^2_{ik}$ and $\Sigma_{\frac{1}{n}}=\frac{1}{n}(\Gamma\Gamma^T)$. For simplicity, we denote the conditional distributions  of $\mathcal{L}(B_{\frac{m}{n}}|B_{\frac{m-1}{n}}=y)$ and  $\mathcal{L}(M_{\frac{m}{n}}|M_{\frac{m-1}{n}}=y)$ with $g(x)$ and $n(x)$, respectively. We can now proceed to examining the logarithms of the latter:
\begin{equation*} \begin{split} 
 &\textstyle \log(g(x)) =\log \Big(\frac{1}{\sqrt{(2\pi)^l \det \Sigma_{\frac{1}{n}}}} \times \\&\textstyle \big(\prod_{i=1}^l\frac{1}{x^i}\big) \exp \big(-\frac{1}{2}(\log x - \mu_{\frac{1}{n}}-\log y)^T\Sigma^{-1}_{\frac{1}{n}}(\log x - \mu_{\frac{1}{n}}-\log y) \big)\Big)\\&\textstyle =-\frac{1}{2}\log (2 \pi)^l \det \Sigma_{\frac{1}{n}} - \sum_{i=1}^l x^i -\frac{1}{2}(\log x - \mu_{\frac{1}{n}}-\log y)^T\Sigma^{-1}_{\frac{1}{n}}(\log x - \mu_{\frac{1}{n}}-\log y),
\end{split}
\end{equation*}
and the last term becomes:
\begin{equation*} \begin{split} 
&\textstyle -\frac{1}{2} tr \Big( \Sigma^{-1}_{\frac{1}{n}}\, (\log x - \mu_{\frac{1}{n}}-\log y)(\log x - \mu_{\frac{1}{n}}-\log y)^T\Big)\\&\textstyle =-\frac{1}{2} tr \Big( \Sigma^{-1}_{\frac{1}{n}}\,\big( \log x \log x^T + \mu_{\frac{1}{n}}\mu_{\frac{1}{n}}^T + \log y \log y^T -2\log x \mu_{\frac{1}{n}}^T -2 \log x \log y ^T \\&\textstyle +2 \mu_{\frac{1}{n}} \log y^T\big)\Big).
\end{split}
\end{equation*}
On the other hand:
\begin{equation*} \begin{split} 
\log(n(x))&=\textstyle \log \Big(\frac{1}{\sqrt{(2 \pi )^l}\det (\frac{1}{n}\mathcal{I})}\exp \big(-\frac{1}{2}(x-y)^T (\frac{1}{n}\mathcal{I})^{-1}(x-y)\big)\Big)\\&\textstyle =-\frac{1}{2}\log (2 \pi )^l  \det (\frac{1}{n}\mathcal{I})-\frac{1}{2}(x-y)^T (\frac{1}{n}\mathcal{I})^{-1}(x-y),
\end{split}
\end{equation*}
and the last term becomes:
\begin{equation*} \begin{split} \textstyle 
-\frac{1}{2}tr\Big((\frac{1}{n}\mathcal{I})^{-1}(x-y)(x-y)^T\Big)=-\frac{1}{2}tr\Big(n \mathcal{I} \big(xx^T-2xy^T+yy^T\big)\Big).
\end{split}
\end{equation*}
Before computing their expectations we look at the one of $M_{\frac{m}{n}}|M_{\frac{m-1}{n}}=y$:
\begin{equation*} \begin{split} \textstyle 
\big(\E[M_{\frac{m}{n}}|M_{\frac{m-1}{n}}=y]\big)_i&=\textstyle e^{\mu_i+\log y +\frac{1}{2}(\Sigma_{\frac{1}{n}})_{ii}}
\\&\textstyle =e^{-\frac{1}{2n}\sum _{k=1}^l \Gamma_{ik}^2 +\log y_i + \frac{1}{2n}\sum _{k=1}^l \Gamma_{ik}^2 }= y_i,
\end{split}
\end{equation*}
as $i=1, \dots, j$, and the variance $Var[M_{\frac{m}{n}}|M_{\frac{m-1}{n}}=y]$ is:
\begin{equation*} \begin{split} 
&=\textstyle  \exp \Big((\mu_{\frac{1}{n}})_i+ \log y_i+\mu_j(\frac{1}{n})+ \log y_j + \frac{1}{2}((\Sigma_{\frac{1}{n}})_{ii}+(\Sigma_{\frac{1}{n}})_{jj})\Big)(e^{(\Sigma_{\frac{1}{n}})_{ij}}-1)\\&\textstyle =\exp \Big( -\frac{1}{2n}\sum _{k=1}^l \Gamma_{ik}^2 + \log y_i -\frac{1}{2n}\sum _{k=1}^l\Gamma_{jk}^2 + \log y_j + \frac{1}{2n}(\sum _{k=1}^l \Gamma_{ik}^2+\sum _{k=1}^l \Gamma_{jk}^2)\Big)\\&\textstyle (e^{\frac{1}{n}\sum _{k=1}^l \Gamma_{ik}\Gamma_{jk}}-1) =y_i y_j(e^{\frac{1}{n}\sum _{k=1}^l \Gamma_{ik}\Gamma_{jk}}-1),
\end{split}
\end{equation*}
as $i,j=1, \dots, l$ while the second moment is:
\begin{equation*} \begin{split} 
&\textstyle \big( \E \big[(M_{\frac{m}{n}}|M_{\frac{m-1}{n}}=y)(M_{\frac{m}{n}}|M_{\frac{m-1}{n}}=y)^T\big]\big)_{ij}\\=& \textstyle \big(\E[M_{\frac{m}{n}}|M_{\frac{m-1}{n}}=y]\, \E[M_{\frac{m}{n}}|M_{\frac{m-1}{n}}=y]^T\big)_{ij} + \big(Var[M_{\frac{m}{n}}|M_{\frac{m-1}{n}}=y]\big)_{ij}\\=&\textstyle (yy^T)_{ij}+ y_i y_j(e^{\frac{1}{n}\sum _{k=1}^l \Gamma_{ik}\Gamma_{jk}}-1)\\=&\textstyle y_i y_j e^{\frac{1}{n}\sum _{k=1}^l \Gamma_{ik}\Gamma_{jk}}.
\end{split}
\end{equation*}
Moreover, we recall that $\log (M_{\frac{m}{n}}|M_{\frac{m-1}{n}}) \sim \mathcal{N}_l (\mu_{1/n}, \Sigma_{1/n})$ and $tr(xy^T)=tr(yx^T) $ $ \forall x, y \in \R^{l\times l}$. Now we can compute the expectations of the conditioned densities: 
\begin{equation*} \begin{split} 
&\textstyle \E_{M_{\frac{m}{n}}|M_{\frac{m-1}{n}}}[\log g(x)]=-\frac{1}{2}\log (2 \pi)^l \det \Sigma_{\frac{1}{n}} - \sum_{i=1}^l( (\mu_{\frac{1}{n}})_i+\log y_i)+\\&\textstyle 
 -\frac{1}{2}tr\Big( \Sigma_{\frac{1}{n}}^{-1}\big(\Sigma_{\frac{1}{n}} + (\mu_{1/n}+\log y)(\mu_{1/n}+\log y)^T +\mu_{1/n}\mu_{1/n}^T+\log y \log y^T\\& \textstyle -2 (\mu_{1/n}+\log y)\mu_{1/n}^T-2(\mu_{1/n}+\log y) \log y^T+ 2 \mu_{1/n}\log y^T \big)\Big)=\\&\textstyle =-\frac{1}{2}\log (2 \pi)^l \det \Sigma_{\frac{1}{n}} - \sum_{i=1}^l ( (\mu_{\frac{1}{n}})_i+\log y_i) -\frac{1}{2} tr \Big( \Sigma_{\frac{1}{n}}^{-1}\big( \Sigma_{1/n}+ \mu_{1/n}\mu_{1/n}^T +\\&\textstyle +\log y \log y^T + 2 \log y\mu_{1/n}^T +\mu_{1/n}\mu_{1/n}^T+ \log y \log y^T -2 \mu_{1/n}\mu_{1/n}^T-2 \log y\mu_{1/n}^T\\&\textstyle  -2\mu_{1/n} \log y^T -2\log y \log y^T +2 \mu_{1/n} \log y^T\big)\Big)=\\ &\textstyle  =-\frac{1}{2}\log (2 \pi)^l \det \Sigma_{\frac{1}{n}} - \sum_{i=1}^l ( (\mu_{\frac{1}{n}})_i+\log y_i) -\frac{1}{2} tr (\mathcal{I}). 
\end{split}
\end{equation*}
Recalling that $(\Sigma_{\frac{1}{n}})_{ij}=\frac{1}{n}(\Gamma \Gamma^T)$ and $(\mu_{\frac{1}{n}})_i=-\frac{1}{2n} \sum_{k=1}^l \Gamma_{ik}^2$, we have: 
\begin{equation*} \begin{split} 
&\textstyle E_{M_{\frac{m}{n}}|M_{\frac{m-1}{n}}}[\log g(x)]=\\&\textstyle = -\frac{l}{2}\log \frac{2 \pi }{n}- \frac{1}{2} \log (\det (\Gamma \Gamma^T)) +\frac{1}{2n}\sum_{i=1}^l \sum_{k=1}^l \Gamma_{ik}^2 - \sum_{i=1}^l \log y_i - \frac{l}{2}\\&\textstyle =-\frac{l}{2}\log \frac{2 \pi }{n} - \frac{1}{2} \log (\det (\Gamma \Gamma^T)) +\frac{1}{2n}\sum_{i=1}^l \sum_{k=1}^l \Gamma_{ik}^2 - \sum_{i=1}^l \log y_i - \frac{l}{2}.
\end{split}
\end{equation*} 
While on the other hand:
\begin{equation*} \begin{split} \textstyle 
E_{M_{\frac{m}{n}}|M_{\frac{m-1}{n}}}[\log n(x)]&\textstyle =-\frac{l}{2}\log \frac{2 \pi }{n}-\frac{1}{2}tr \Big(n \mathcal{I}\big(yy^T e^{\Sigma_{1/n}} -2 y y^T+ yy^T\big)\Big)\\&\textstyle =-\frac{l}{2}\log \frac{2 \pi }{n} -\frac{1}{2}tr \big (n \mathcal{I} yy^T e^{\Sigma_{1/n}} \big)+\frac{1}{2}y^Tn \mathcal{I}y\\&\textstyle =-\frac{l}{2}\log \frac{2 \pi }{n} -\frac{n}{2} \sum_{i=1}^l (y_i)^2 e^{\Sigma_{ii}}+ \frac{n}{2} \sum_{i=1}^l (y_i)^2 \\& \textstyle =-\frac{l}{2}\log \frac{2 \pi }{n} -\frac{n }{2}\sum_{i=1}^l (y_i)^2 (e^{\frac{1}{n}\sum_{k=1}^l\Gamma_{ik}^2}-1).
\end{split}
\end{equation*}
Thus, the relative entropy is:
\begin{equation*} \begin{split} 
&\textstyle H(\mathcal{L}(M_{\frac{m}{n}}|M_{\frac{m-1}{n}}=y)|\mathcal{L}(B_{\frac{m}{n}}|B_{\frac{m-1}{n}}=y))\\&\textstyle =\mathbb{E}_{\mathcal{L}(M_{\frac{m}{n}}|M_{\frac{m-1}{n}}=y)}\big[\log (g(x))-\log(n(x))\big]\\
&\textstyle =-\frac{l}{2}\log \frac{2 \pi }{n} - \frac{1}{2} \log (\det (\Gamma \Gamma^T)) +\frac{1}{2n}\sum_{i=1}^l \sum_{k=1}^l \Gamma_{ik}^2 - \sum_{i=1}^l \log y_i - \frac{l}{2}+\\&\textstyle +\frac{l}{2}\log \frac{2 \pi }{n} +\frac{n }{2}\sum_{i=1}^l (y_i)^2 (e^{\frac{1}{n}\sum_{k=1}^l\Gamma_{ik}^2}-1)\\&\textstyle =- \frac{1}{2} \log (\det (\Gamma \Gamma^T)) +\frac{1}{2n}\sum_{i=1}^l \sum_{k=1}^l \Gamma_{ik}^2 - \sum_{i=1}^l \log y_i - \frac{l}{2}+\\& \textstyle +\frac{n }{2}\sum_{i=1}^l (y_i)^2 (e^{\frac{1}{n}\sum_{k=1}^l\Gamma_{ik}^2}-1).
\end{split}
\end{equation*}
Now we want to integrate this quantity:
\begin{equation*} \begin{split} 
&\textstyle \int_y H(\mathcal{L}(M_{\frac{m}{n}}|M_{\frac{m-1}{n}}=y)|\mathcal{L}(B_{\frac{m}{n}}|B_{\frac{m-1}{n}}=y))\mathcal{L}(M_{\frac{m-1}{n}})dy=\\&\textstyle =\int_y - \frac{1}{2} \log (\det (\Gamma \Gamma^T))+\frac{1}{2n}\sum_{i=1}^l \sum_{k=1}^l \Gamma_{ik}^2 - \sum_{i=1}^l \log y_i - \frac{l}{2}+\\&\textstyle +\frac{n }{2}\sum_{i=1}^l (y_i)^2 (e^{\frac{1}{n}\sum_{k=1}^l\Gamma_{ik}^2}-1) \mathcal{L}(M_{\frac{k-1}{n}})dy=\\
&\textstyle =- \frac{1}{2} \log (\det (\Gamma \Gamma^T)) +\frac{1}{2n}\sum_{i=1}^l \sum_{k=1}^l \Gamma_{ik}^2 - \frac{l}{2}+\\&\textstyle  - \sum_{i=1}^l  \big(\E[\log M_{\frac{m-1}{n}}]\big)_i+\frac{n }{2}\sum_{i=1}^l \big(\E[(M_{\frac{m-1}{n}})^2]\big)_i\big(e^{\frac{1}{n}\sum_{k=1}^l\Gamma_{ik}^2}-1\big).\end{split}
\end{equation*}
Recalling that $\big(\E[\log M_{\frac{m-1}{n}}]\big)_i=\mu_i(\frac{m-1}{n})=-\frac{m-1}{2n}\sum_{k=1}^l \Gamma_{ik}^2$ and $\big(\E[M_{\frac{m-1}{n}}M_{\frac{m-1}{n}}^T]\big)_{ij}$ $=e^{\Sigma_{ij}(\frac{m-1}{n})}=e^{\frac{m-1}{n} \sum_{k=1}^l \Gamma_{ik}\Gamma_{jk}}$, we have: 
\begin{equation} \begin{split} \label{eq_int_srE}
&\textstyle \int_y H(\mathcal{L}(M_{\frac{m}{n}}|M_{\frac{m-1}{n}}=y)|\mathcal{L}(B_{\frac{m}{n}}|B_{\frac{m-1}{n}}=y))\mathcal{L}(M_{\frac{m-1}{n}})dy\\
=&\textstyle - \frac{1}{2} \log (\det (\Gamma \Gamma^T))+\frac{m}{2n}\sum_{i=1}^l \sum_{k=1}^l \Gamma_{ik}^2 - \frac{l}{2}+\frac{n }{2}\sum_{i=1}^l e^{\frac{m-1}{n} \sum_{k=1}^l \Gamma_{ik}^2}\big(e^{\frac{1}{n}\sum_{k=1}^l\Gamma_{ik}^2}-1\big).
\end{split}
\end{equation}
Summing up we have that: 
\begin{equation*}\begin{split} 
&\textstyle H(M| B)|_{F^n}= H(\mathcal{L}(M_{\frac{1}{n}})|\mathcal{L}(B_{\frac{1}{n}}))+\\&\textstyle +\sum_{m=2}^{n} \int H(\mathcal{L}(M_{\frac{m}{n}}|M_{\frac{m-1}{n}}=y)|\mathcal{L}(B_{\frac{m}{n}}|B_{\frac{m-1}{n}}=y))\mathcal{L}(M_{\frac{m-1}{n}})dy.
\end{split} \end{equation*}
Thus by Equations \eqref{ex_extra_1/n} and \eqref{eq_int_srE}, 
\begin{equation*} \begin{split} 
=&\textstyle \frac{1}{2}\Big( - \log (\det (\Gamma \Gamma^T)) + \frac{1}{n} \sum_{i=1}^l \sum_{k=1}^l \Gamma_{ik}^2 -l +n  \sum_{i=1}^l \big(e^{\frac{1}{n}\sum
_{k=1}^l\Gamma_{ik}^2}-1 \big) \Big) +\\&\textstyle + \sum _{m=2}^n \Big(- \frac{1}{2} \log (\det (\Gamma \Gamma^T))+\frac{m}{2n}\sum_{i=1}^l \sum_{k=1}^l \Gamma_{ik}^2 - \frac{l}{2}+\\&\textstyle +\frac{n }{2}\sum_{i=1}^l e^{\frac{m-1}{n} \sum_{k=1}^l \Gamma_{ik}^2}\big(e^{\frac{1}{n}\sum_{k=1}^l\Gamma_{ik}^2}-1\big) \Big).
\end{split}
\end{equation*}
For the specific relative entropy we look at the asymptotic behaviour of each term multiplied by $\frac{1}{n}$. Regarding the first term we have:
\begin{equation*} \begin{split} 
&\textstyle \frac{1}{2n}\Big( - \log (\det (\Gamma \Gamma^T)) + \frac{1}{n} \sum_{i=1}^l \sum_{k=1}^l \Gamma_{ik}^2 -l +n  \sum_{i=1}^l \big(e^{\frac{1}{n}\sum
_{k=1}^l\Gamma_{ik}^2}-1 \big) \Big) \xlongrightarrow {n \rightarrow \infty }0,
\end{split}
\end{equation*}
while for the other terms:
\begin{equation*} \begin{split} 
&\textstyle \frac{1}{n}\sum _{m=2}^n \Big(- \frac{1}{2} \log (\det (\Gamma \Gamma^T)) +\frac{m}{2n}\sum_{i=1}^l \sum_{k=1}^l \Gamma_{ik}^2 - \frac{l}{2}+\\&\textstyle +\frac{n }{2}\sum_{i=1}^l e^{\frac{m-1}{n} \sum_{k=1}^l \Gamma_{ik}^2}\big(e^{\frac{1}{n}\sum_{k=1}^l\Gamma_{ik}^2}-1\big) \Big)=\\&\textstyle =
-\frac{n-1}{n} \frac{1}{2} \log (\det (\Gamma \Gamma^T))  + \frac{n(n+1)-1}{4n^2}\sum_{i=1}^l \sum_{k=1}^l \Gamma_{ik}^2 -\frac{n-1}{n}\frac{l}{2} \\&\textstyle + \frac{n}{n}\frac{1}{2}(e^{\frac{1}{n}\sum_{k=1}^l\Gamma_{ik}^2}-1\big)  \big ( \frac{e^{\sum_{k=1}^l\Gamma_{ik}^2}-e^{\frac{1}{n}\sum_{k=1}^l\Gamma_{ik}^2}}{e^{\frac{1}{n}\sum_{k=1}^l\Gamma_{ik}^2}-1}\big)\\&\textstyle 
\xlongrightarrow {n \rightarrow \infty } - \frac{1}{2} \log (\det (\Gamma \Gamma^T))  +\frac{1}{4} \sum_{i=1}^l \sum_{k=1}^l \Gamma_{ik}^2- \frac{l}{2}+ \sum_{i=1}^l \frac{e ^{\sum_{k=1}^l\Gamma_{ik}^2}-1}{2}.
\end{split}
\end{equation*}
In conclusion we get the specific relative entropy: 
\begin{equation*} \begin{split} 
 h_l(\mathcal L(M^{\Gamma})|\mathbb B^l)&\textstyle =\lim_{n \rightarrow \infty }\frac{1}{n}H(\mathcal L(M^{\Gamma})|\mathbb B^l)|_{\mathcal F^n}\\ &\textstyle =0+ \big( - \frac{1}{2} \log (\det (\Gamma \Gamma^T))  +\frac{1}{4} \sum_{i=1}^l \sum_{k=1}^l \Gamma_{ik}^2- \frac{l}{2}+ \frac{e-1}{2} \big)\\&\textstyle =- \frac{1}{2} \log (\det (\Gamma \Gamma^T))  +\frac{1}{4} \sum_{i=1}^l \sum_{k=1}^l \Gamma_{ik}^2- \frac{l}{2}+\sum_{i=1}^l \frac{e ^{\sum_{k=1}^l\Gamma_{ik}^2} -1}{2}.
\end{split}
\end{equation*}}
\end{proof}

\begin{rmk}\label{rem:eq_first_ex}
{We say that Gantert's inequality is an equality for $M$ if   
    \begin{equation*} \begin{split} 
h(\mathcal L(M)|\mathbb B^l)&\textstyle =\frac{1}{2} \mathbb{E} \Big[ \int _0^1 \big[tr( \Sigma_t \Sigma_t^T) -l -\log (\det ( \Sigma_t \Sigma_t^T  ))\big]dt\Big], 
\end{split}
\end{equation*}
whenever $dM_t=\Sigma_t dB_t$ and $M_0=B_0$.
In the present case, since $dM_t^{\Gamma}:=diag(M_t^{\Gamma})\Gamma dB_t$, we have that $\ \Sigma_t=diag(M_t^{\Gamma})\Gamma$. Specifically, omitting the superscript $\Gamma$, we have that  \[\textstyle  \Sigma_t \Sigma_t^T =diag (M_t) (\Gamma \Gamma ^T)diag (M_t).\] Therefore $tr(\Sigma_t \Sigma_t^T )= \sum_{i=1}^l (M_t^i)^2 \sum _{k=1}^l \Gamma_{ik}^2$ and \[\textstyle \det ( \Sigma_t \Sigma_t^T ) = (\det (diag(M_t)))^2 \det (\Gamma \Gamma^T)= \prod _{i=1}^l (M_t^i)^2 \det (\Gamma \Gamma^T).\]
Therefore, by Fubini, we would expect to find: 
\begin{equation*} \begin{split} \textstyle 
h_l(M^{\Gamma}|B)&\textstyle =\frac{1}{2} \Big(\int^1_0 \sum_{i=1}^l \E[(M_t^i)^2] \sum _{k=1}^l \Gamma_{ik}^2-l - \sum_{i=1}^l \E [(\log M_t^i)^2]- \log \det (\Gamma \Gamma^T)\Big)dt
\end{split}
\end{equation*}
Recalling that $\log M_t \sim \mathcal{N}_l(\mu_t, \sigma^2_t)$ where $(\mu_t)_i=-\frac{t}{2}\sum _{k=1}^l \Gamma_{ik}^2$ and $(\sigma^2_{t})_{ij}=t(\Gamma \Gamma^T)$, we have that $\forall x \in \R^l$: 
\begin{equation*} \begin{split} 
&\E[M_tM_t^T]_{ij}=(\E[(M_t)]\E[(M_t)]^T)_{ij}+ Var[M_t]_{ij}=e^{t\sum _{k=1}^l \Gamma_{ik}\Gamma_{jk}}
\end{split}
\end{equation*}
Therefore, we would expect
\begin{equation*} \begin{split} 
h_l(M^{\Gamma}|B)&=\textstyle \frac{1}{2}\int^1_0\Big(  \sum_{i=1}^l e^{\sum_{k=1}^l \Gamma_{ik}^2} \sum _{k=1}^l \Gamma_{ik}^2-l +t\sum_{i=1}^l \sum _{k=1}^l \Gamma_{ik}^2 - \log \det (\Gamma \Gamma^T) \Big)dt\\&=\textstyle \frac{1}{2}\Big(\sum_{i=1}^l e^{t\sum_{k=1}^l \Gamma_ik^2} -l t +  \frac{t^2}{2} \sum_{i=1}^l \sum _{k=1}^l \Gamma_{ik}^2 -t   \log \det (\Gamma \Gamma^T)\Big|_0^1\\&=\textstyle \frac{1}{2}\Big(\sum_{i=1}^l \big(e^{\sum_{k=1}^l \Gamma_ik^2}-1\big)-l+\frac{1}{2} \sum_{i=1}^l \sum _{k=1}^l \Gamma_{ik}^2 - \log \det (\Gamma \Gamma^T) \Big),
\end{split}
\end{equation*}
which is exactly what  Lemma \ref{lem:Black-Scholes_vs_BM} states. Hence Gantert's inequality is an equality in the present case.}
\end{rmk}

\subsection{Specific Relative Entropy between two multivariate Black-Scholes models}\label{subsub4}
We proceed to the proof of Lemma \ref{lem:Black-Scholes} from the introduction, which we rephrase here for the reader's convenience:
\begin{lem}\label{lem3.5}
We have {
\begin{equation}\label{equazioneB}\textstyle 
h_l(\mathcal L(M^{\Gamma_1})|\mathcal L(M^{\Gamma_2}))=\frac{1}{2}\Big(tr\big((\Gamma_2 \Gamma_2^T)^{-1}(\Gamma_1 \Gamma_1^T)\big)-l - \log \frac{\det (\Gamma_1 \Gamma_1^T)}{\det (\Gamma_2 \Gamma_2^T)}\Big) .
\end {equation}}
\end{lem}
\begin{proof}
We have:
\begin{equation*} \begin{split} \textstyle 
H   (\mathcal L(M^{\Gamma_1})|\mathcal L(M^{\Gamma_2}))|_{\mathcal{F}^n} = H\Big( \mathcal{L} \big( M_{1/n}^{\Gamma_1}\!,\, M_{\frac{2}{n}}^{\Gamma_1}\!,\, \cdots\!,\, M_{1}^{\Gamma_1} \big) | \mathcal{L} \big( M_{1/n}^{\Gamma_2}\!,\, M_{\frac{2}{n}}^{\Gamma_2}\!,\, \cdots\!,\, M_{1}^{\Gamma_2} \big) \Big),
\end{split}\end{equation*}
and the r.h.s.\ is given explicitly by
\begin{equation*} \begin{split}\textstyle 
&H\Big( \mathcal{L} \big( M^{1 \, \Gamma_1}_{\frac{1}{n}}, \dots, M^{l \, \Gamma_1}_{1/n}, M^{1 \, \Gamma_1}_{\frac{2}{n}}, \dots, M^{l \, \Gamma_1}_{\frac{2}{n}}, \dots \big) | \mathcal{L}   \big( M^{1 \, \Gamma_2}_{\frac{1}{n}}, \dots, M^{l \, \Gamma_2}_{1/n}, M^{1 \, \Gamma_2}_{\frac{2}{n}}, \dots, M^{l \, \Gamma_2}_{\frac{2}{n}},\\& \dots \big) \Big).
\end{split}\end{equation*}
Applying Lemma \ref{lem1.2}, this is equal to
\begin{equation*} \begin{split}
&\textstyle =H\Bigg( \mathcal{L} \Big( M^{1, \, \Gamma_1}_{\frac{1}{n}}, \dots, M^{l, \, \Gamma_1}_{\frac{1}{n}}, \frac{M^{1, \, \Gamma_1}_{2/n}}{M^{1 \, \Gamma_1}_{1/n}}, \dots, \frac{M^{2, \, \Gamma_1}_{\frac{2}{n}}}{M^{2 \, \Gamma_1}_{\frac{1}{n}}} ,\dots \Big) | \mathcal{L} \Big( M^{1, \, \Gamma_2}_{\frac{1}{n}}, \dots, M^{l, \, \Gamma_2}_{\frac{1}{n}}, \\&\textstyle ,\frac{M^{1, \, \Gamma_2}_{2/n}}{M^{1, \, \Gamma_2}_{1/n}},\dots, \frac{M^{2,\, \Gamma_2}_{\frac{2}{n}}}{M^{2, \, \Gamma_2}_{\frac{1}{n}}} ,\dots \Big) \Bigg),
\end{split}\end{equation*}
and by Lemma \ref{lem1.3}, it follows
\begin{equation*} \begin{split}
&\textstyle H(\mathcal L(M^{\Gamma_1})|\mathcal L(M^{\Gamma_2}))|_{\mathcal{F}^n} = H\Big( \mathcal{L} ( M^{1, \, \Gamma_1}_{\frac{1}{n}}, \ldots, M^{l, \, \Gamma_1}_{\frac{1}{n}} ) \Big| \mathcal{L} \Big( M^{1, \, \Gamma_2}_{\frac{1}{n}}, \ldots, M^{l, \, \Gamma_2}_{\frac{1}{n}} \Big) +\\
&\textstyle +\int_{y_1, \ldots, y_l} H\Bigg( \mathcal{L} \Big( \frac{M^{1, \, \Gamma_1}_{\frac{2}{n}}}{M^{1 \, \Gamma_1}_{\frac{1}{n}}}, \ldots, \frac{M^{l, \, \Gamma_1}_{\frac{2}{n}}}{M^{l, \, \Gamma_1}_{\frac{1}{n}}} \Big| M^{1, \, \Gamma_1}_{\frac{1}{n}}=y_1, \ldots, M^{l, \, \Gamma_1}_{\frac{1}{n}}=y_l \Big) \Big| \\& \textstyle \mathcal{L} \Big( \frac{M^{1, \, \Gamma_2}_{\frac{2}{n}}}{M^{1, \, \Gamma_2}_{\frac{1}{n}}}, \ldots, \frac{M^{l, \, \Gamma_2}_{\frac{2}{n}}}{M^{l, \, \Gamma_2}_{\frac{1}{n}}} \Big| M^{1, \, \Gamma_2}_{\frac{1}{n}}=y_1, \ldots, M^{l, \, \Gamma_2}_{\frac{1}{n}}=y_l \Big) \Bigg) d\mathcal{L}(M^{1, \, \Gamma_1}_{\frac{1}{n}}, \ldots, M^{l, \, \Gamma_1}_{\frac{1}{n}})\\&(y_1, \ldots, y_l) +\int_{y_1, \ldots, y_l} H\Bigg( \mathcal{L}\Big( \frac{M^{1, \, \Gamma_1}_{\frac{3}{n}}}{M^{1, \, \Gamma_1}_{\frac{2}{n}}}, \ldots, \frac{M^{l, \, \Gamma_1}_{\frac{3}{n}}}{M^{l, \, \Gamma_1}_{\frac{2}{n}}} \Big| M^{1, \, \Gamma_1}_{\frac{1}{n}}=y_1, \ldots, M^{l, \, \Gamma_1}_{\frac{1}{n}}=y_l, \\&\textstyle M^{1, \, \Gamma_1}_{\frac{2}{n}}=z_1, \ldots, M^{l, \, \Gamma_1}_{\frac{2}{n}}=z_l \Big)   \Big |\mathcal{L}\Big( \frac{M^{1, \, \Gamma_2}_{\frac{3}{n}}}{M^{1, \, \Gamma_2}_{\frac{2}{n}}}, \ldots, \frac{M^{l, \, \Gamma_2}_{\frac{3}{n}}}{M^{l, \, \Gamma_2}_{\frac{2}{n}}} \Big| M^{1, \, \Gamma_2}_{\frac{1}{n}}=y_1, \ldots, M^{l, \, \Gamma_2}_{\frac{1}{n}}=y_l, \\&M^{1, \, \Gamma_2}_{\frac{2}{n}}=z_1, \ldots, M^{l, \, \Gamma_2}_{\frac{2}{n}}=z_l \Big) \Bigg)d\mathcal{L}(M^{1, \, \Gamma_1}_{\frac{1}{n}}, \ldots, M^{l \, \Gamma_1}_{\frac{1}{n}})(y_1, \ldots, y_l) d \mathcal{L}\Big( M^{1, \, \Gamma_1}_{\frac{2}{n}}, \ldots, \\&\textstyle M^{l, \, \Gamma_1}_{\frac{2}{n}} \Big)(z_1, \ldots, z_l) + \ldots
\end{split}\end{equation*}
Note that, by the Markov property and the independence of the increments of the Brownian motion, the ratios are all independents and so we can ignore the conditioning. Therefore we want to compute the first term and for $m=2, \dots, n$:
\begin{equation}\label{equazi11} \begin{split}\textstyle 
H\Bigg( \mathcal{L}\Big( \frac{M^{1, \, \Gamma_1}_{\frac{m}{n}}}{M^{1, \, \Gamma_1}_{\frac{m-1}{n}}}, \ldots, \frac{M^{l, \, \Gamma_1}_{\frac{m}{n}}}{M^{l, \, \Gamma_1}_{\frac{m-1}{n}}} \Big) \Big| \mathcal{L} \Big(\frac{M^{1, \, \Gamma_2}_{\frac{m}{n}}}{M^{1, \, \Gamma_2}_{\frac{m-1}{n}}}, \ldots, \frac{M^{l, \, \Gamma_2}_{\frac{m}{n}}}{M^{l, \, \Gamma_2}_{\frac{m-1}{n}}} \Big) \Bigg).
\end{split}\end{equation}
Now, dropping the subscript "$1,2$" for a moment, we look at the distributions of each coordinate $i=1, \dots l$:
\begin{equation*} \begin{split} \textstyle 
\frac{M^i_{\frac{m}{n}}}{M^i_{\frac{m-1}{n}}}&\textstyle =e^{\sum_{k=1}^l\Gamma_{ik}\Delta B_{1/n}^k-\frac{1}{2n}\sum _{k=1}\Gamma _{ik}^2}\\
&\textstyle \sim lognormal(-\frac{1}{2n}\sum_{k=1}^{l}\Gamma_{ik}^2,\frac{1}{n}\sum_{k=1}^l\Gamma _{ik}^2).
\end{split}\end{equation*}
So equation \eqref{equazi11} becomes:
\begin{equation*} \begin{split}
=& \textstyle H\Bigg( \mathcal{L} \Big( e^{\sum_{k=1}^l\Gamma_{1, \,1k}B_{1/n}^k-\frac{1}{2n}\sum_{k=1}^l\Gamma_{1, \,1k}^2}, \ldots, e^{\sum_{k=1}^l\Gamma_{1, \,lk}B_{1/n}^k-\frac{1}{2n}\sum_{k=1}^l\Gamma_{1, \,lk}^2} \Big) \Big| \\&\textstyle \mathcal{L} \Big( e^{\sum_{k=1}^l\Gamma_{2, \,1k}B_{1/n}^k-\frac{1}{2n}\sum_{k=1}^l\Gamma_{2, \,1k}^2}, \ldots, e^{\sum_{k=1}^l\Gamma_{2, \,lk}B_{1/n}^k-\frac{1}{2n}\sum_{k=1}^l\Gamma_{2, \,lk}^2} \Big) \Bigg).
\end{split}\end{equation*}
and again by Lemma \ref{lem1.2}:
\begin{equation}\label{equaz12} \begin{split}
&=\textstyle  H\Bigg( \mathcal{L} \Big( \sum_{k=1}^l\Gamma_{1, \,1k}B_{1/n}^k-\frac{1}{2n}\sum_{k=1}^l\Gamma_{1, \,1k}^2, \ldots, \sum_{k=1}^l\Gamma_{1, \,lk}B_{1/n}^k-\frac{1}{2n}\sum_{k=1}^l\Gamma_{1, \,lk}^2 \Big) \Big| \\
&\textstyle \mathcal{L} \Big( \sum_{k=1}^l\Gamma_{2, \,1k}B_{1/n}^k-\frac{1}{2n}\sum_{k=1}^l\Gamma_{2, \,1k}^2, \ldots, \sum_{k=1}^l\Gamma_{2, \,lk}B_{1/n}^k-\frac{1}{2n}\sum_{k=1}^l\Gamma_{2, \,lk}^2 \Big) \Bigg).
\end{split}\end{equation}
Now in any coordinate there is a normal distributions so we want to compute the relative entropy between two multivariate Gaussians: $H\big(\mathcal{N}_l(\mu_1,\Sigma_1)|\mathcal{N}_l(\mu_2,\Sigma_2)\big)$ where $\mu_{1,2}$ and $\Sigma_{1,2}$ are respectively the mean vector and the covariance matrix given by the coordinates of \eqref{equaz12}: $(\Sigma_{1,2})_{ij}=\frac{1}{n}\Sigma_{k=1}^l(\Gamma_{1,2})_{ik}(\Gamma_{1,2})_{ jk}$  and  $\mu_{1,2}= -\frac{1}{2n} \big(\sum _{k=1}^l(\Gamma_{1,2,})_{ 1k}^2, \dots, \sum _{k=1}^l(\Gamma_{1,2})_{ lk}^2\big)^T$. To show how to get the covariance matrix, we first define $Z_{i}:= \sum_{k=1}^l\Gamma_{i}B_{1/n}^k-\frac{1}{2n}\sum _{k=1}\Gamma_{i}^2$ and $\bar{Z}_{i}:= \sum_{k=1}^l\Gamma_{i}B_{1/n}^k$. Ignoring  the subscripts for simplicity we have: \begin{equation*} \begin{split} (\Sigma)_{ij}&=Cov[(Z_i,Z_j)]=\mathbb{E}\big[ (Z_i-\mathbb{E}[Z_i])(Z_j-\mathbb{E}[Z_j])\big] \\ 
&=\textstyle  \mathbb{E}\big[ \bar{Z_i} \bar{Z_j}\big]= \mathbb{E}\big[ \sum_{k=1}^l\Gamma_{ik}B_{1/n}^k \sum_{k=1}^l\Gamma_{jk}B_{1/n}^k\big]\\
&=\textstyle  \sum_{k=1}^l\Gamma_{ik}\Gamma_{jk} \mathbb{E}\big[ (B_{1/n}^k)^2\big] = \frac{1}{n} \sum_{k=1}^l\Gamma_{ik}\Gamma_{jk}{=\frac{1}{n} \big(\Gamma \Gamma^T\big)_{ij}}.
\end{split}\end{equation*}
After noticing that we can argue in the same way for the first term since $M_{\frac{1}{n}}^{i, \, \Gamma_{1,2}} \sim lognormal(-\frac{1}{2n}\sum_{k=1}^{l}(\Gamma_{1,2})_{ik}^2,\frac{1}{n}\sum_{k=1}^l(\Gamma_{1,2})_{ik}^2) $ then by Lemma \ref{Lemma2.1} and using Example \ref{ex1.1} we have: 
\begin{equation*} \label{equaz13}  \begin{split} \textstyle 
    H\big( \mathcal L(M^{\Gamma_1})|\mathcal L(M^{\Gamma_2}) \big)|_{\mathcal{F}^n}&\textstyle =\sum _{m=1}^{n}H(\mathcal{N}_l(\mu_1,\Sigma_1)|\mathcal{N}_l(\mu_2,\Sigma_2))\\&\textstyle =nH(\mathcal{N}_l(\mu_1,\Sigma_1)|\mathcal{N}_l(\mu_2,\Sigma_2))
    \\&\textstyle = \frac{n}{2} (tr(\Sigma_2^{-1}\Sigma_1)+(\mu_2-\mu_1)^T\Sigma_2^{-1}(\mu_2-\mu_1)-l - \log \frac{\det \Sigma_1}{\det \Sigma_2}). \end{split}\end{equation*} 
    Looking at the asymptotic behaviour of the terms, we expect that:
    \begin{equation*}\begin{split} \textstyle 
    tr(\Sigma_2^{-1}\Sigma_1) \, \text{ and } \, \log ( \frac{\det \Sigma_2}{\det \Sigma_2}) \, \sim \frac{n}{n},
    \end{split}\end{equation*} 
    while for the quadratic term:
    \begin{equation*} \begin{split}\textstyle 
    (\mu_2-\mu_1)^T\Sigma_2^{-1}(\mu_2-\mu_1) \sim \frac{1}{n^2} \, n.
    \end{split}\end{equation*} 
    In this way, we get the specific relative entropy desired:
    \begin{equation*}\begin{split} h_{{l}}(\mathcal L(M^{\Gamma_1})|\mathcal L(M^{\Gamma_2}))&=\textstyle \lim _{n \rightarrow \infty} \frac{1}{n} H(\mathcal L(M^{\Gamma_1})|\mathcal L(M^{\Gamma_2}))|_{\mathcal{F}^n}\\&\textstyle {=\frac{1}{2}\Big(tr\big((\Gamma_2 \Gamma_2^T)^{-1}(\Gamma_1 \Gamma_1^T)\big)-l - \log \frac{\det (\Gamma_1 \Gamma_1^T)}{\det (\Gamma_2 \Gamma_2^T)}\Big)}.
    \end{split}\end{equation*} 
\end{proof}

\begin{rmk} \label{rmk12}
    In Lemma \ref{lem3.5} we have that $\Sigma_{u}=\frac{1}{n}\Gamma_{u}\Gamma_{u}^T$ with $u=1,2$. Indeed, $ (\frac{1}{n}\Gamma_{u}\Gamma_{u}^T)_{ij} =\frac{1}{n} \sum_{k=1}^l (\Gamma_u) _{ik} (\Gamma_u) _{kj}^T= \frac{1}{n}\sum_{k=1}^l(\Gamma_u) _{ik} (\Gamma_u) _{jk} = (\Sigma_{u})_{ij}$.
\end{rmk}

\begin{cor}
If $\Gamma_1=diag(\gamma_i)_{i=1,\dots,l}$ and $\Gamma_2=diag(\bar{\gamma_i})_{i=1,\dots,l}$  then  $$\textstyle h_{{l}}(M^{\Gamma_1}|M^{\Gamma_2})=\frac{1}{2}\big(\sum_{i=1}^l \frac{\gamma_i^2}{\bar{\gamma}_i^2}-l-\sum_{i=1}^l \log\frac{\gamma_i^2}{\bar{\gamma}_i^2} \big).$$
\end{cor}

\begin{proof}
    Follows from the previous lemma. 
\end{proof}

\begin{cor}
    If $\Gamma_1$ and $\Gamma_2$ are triangular matrices then also in this case $$\textstyle h_{{l}}(M^{\Gamma_1}|M^{\Gamma_2})=\frac{1}{2}\big(\sum_{i=1}^l \frac{\gamma_i^2}{\bar{\gamma}_i^2}-l-\sum_{i=1}^l \log\frac{\gamma_i^2}{\bar{\gamma}_i^2} \big),$$ where $(\gamma_i)_{i=1,\dots,l}$ are the diagonal elements of $\Gamma_1$  and $(\bar{\gamma}_i)_{i=1,\dots,l}$ are the diagonal elements of $\Gamma_2$. 
    \end{cor}
    \begin{proof}
        We can suppose that $\Gamma_1$ and $\Gamma_2$ are upper triangular matrices, although the proof reveals that the two matrices can be also both lower triangular or of a different type of triangularity.  {Define $\bar{\Sigma}_1:=\Gamma_1 \Gamma_1^T$ and $\bar{\Sigma}_2:=\Gamma_2\Gamma_2^T$. Thus, we have: 
 \begin{equation*} \begin{split}
    \det \bar{\Sigma}_1&\textstyle =\prod _{i=1}^l \gamma_i^2.
 \end{split}\end{equation*}}
Similarly
 $\det \bar{\Sigma}_2^{-1}= \frac{1}{\prod_{i=1}^l\bar{\gamma}_i^2}$.
  Therefore
{ \begin{equation*} \begin{split}\textstyle 
 \log \frac{\det \bar{\Sigma}_1}{\det \bar{\Sigma_2}}&\textstyle =\sum_{i=1}^l \log \frac{\gamma_i^2}{\bar{\gamma}_i^2}.
 \end{split}\end{equation*}}
To compute the trace of $\bar{\Sigma}_2^{-1}\bar{\Sigma}_1$ we need some more  work. Firstly,
{\begin{equation*} \begin{split}
 tr(\bar{\Sigma}_2^{-1}\bar{\Sigma}_1)&=tr\big((\Gamma_2\Gamma_2^T)^{-1}(\Gamma_1 \Gamma_1^T)\big)
 =tr\big((\Gamma_2^T)^{-1} (\Gamma_2)^{-1}\,\Gamma_1\Gamma_1^T\big),
 \end{split}\end{equation*}}
where $(\Gamma_2)^{-1}$ and $\Gamma_1$ are upper triangular matrices while 
 $(\Gamma_2^T)^{-1}$ and $\Gamma_1^T$ are lower triangular matrices. Moreover, by the cyclic property of the trace: $tr(\bar{\Sigma}_2^{-1}\bar{\Sigma}_1)= tr\big( (\Gamma_2)^{-1}\,\Gamma_1\Gamma_1^T\, (\Gamma_2^T)^{-1}\big)$ where  $(\Gamma_2)^{-1}\,\Gamma_1$ is an upper triangular matrix and $\Gamma_1^T\, (\Gamma_2^T)^{-1}$ is a lower triangular matrix.
 
 Now note that if $U,L$ are respectively upper and lower square $l$-matrix then:
 \begin{equation*} \begin{split}
 tr(UL)&\textstyle =\sum_{i=1}^l(UL)_{ii}
 =\sum_{i=1}^l (\sum_k^l U_{ik}L_{ki})
 =\sum _{i=1}^l U_{ii}L_{ii}.
 \end{split}\end{equation*}
 and if $X,Y$ are two triangular square $l$-matrix of the same type then $(XY)_{ii}=X_i Y_i$ for all $i=1,\dots, l$. 
 Therefore, in our case we have obtained the first term in \eqref{equazioneB}:
 \begin{equation*} \begin{split}
 {tr(\bar{\Sigma}_2^{-1}\bar{\Sigma}_1)}&=\textstyle \sum _{i=1}^l \big((\Gamma_2)^{-1} \Gamma_1\big)_{ii}\,  \big(\Gamma_1^T (\Gamma_2^T)^{-1}\big)_{ii} \\&\textstyle =\sum _{i=1}^l ((\Gamma_2)^{-1})_{ii} (\Gamma_1)_{ii} (\Gamma_1^T)_{ii} ((\Gamma_2^T)^{-1})_{ii}
 =\sum _{i=1}^l \frac{\gamma_i^2}{\bar{\gamma}_i^2},
 \end{split}\end{equation*}
 where $(\gamma_i)_{i=1,\dots,l}, (\bar{\gamma}_i)_{i=1,\dots,l}$ are respectively the diagonal elements of $\Gamma_1$ and $\Gamma_2$.  \end{proof}

\begin{rmk}\label{rem:eq_second_ex}
    Let $(\Sigma_t)_t, (\bar{\Sigma}_t)_t $ be $\R^{l\times l}$-valued adapted processes. Let $B_t \in \R^l$ multivariate Brownian motion and $M_t, \bar{M}_t$ such that $dM_t=\Sigma_tdB_t$ and $d\bar{M}_t=\bar \Sigma_t dB_t$. Then one can hope to approximate  $h(\mathcal L(M)|\mathcal L(\bar{M}))$ with
    \begin{equation*} \begin{split}\textstyle 
    \frac{1}{2} \mathbb{E} \Big[ \int _0^1 \big [tr((\bar{\Sigma}_t \bar{\Sigma}_t^T)^{-1} \Sigma_t \Sigma_t^T) -l -\log \det ((\bar{\Sigma}_t \bar{\Sigma}_t^T)^{-1} \Sigma_t \Sigma_t^T ) \big]dt\Big],
    \end{split}\end{equation*}
    and in fact we would hope that the latter is a lower bound of the former (if we extrapolate Gantert's inequality).
     In our case if we define $\Lambda_t:=diag(M_t)\Gamma_1 $ and $\bar{\Lambda}_t:=diag(\bar M_t)\Gamma_2$ both in $\mathbb{R}^{l \times l}$ then this becomes: \begin{equation} \label{equaz14}\begin{split}\textstyle  \mathbb{E} \Big[ \int _0^1 \frac{1}{2}\big(tr((\bar{\Lambda}_t\bar{\Lambda}_t^T)^{-1}\Lambda_t \Lambda_t^T ) -l -\log \det((\bar{\Lambda}_t\bar{\Lambda}_t^T)^{-1}\Lambda_t \Lambda_t^T) \big) dt \Big].\end{split}\end{equation}
    Since,  $(\Lambda_t)_{ij}=m_i(\Gamma_1)_{ij}$, where $m_i:=(M_t)_{ii}$ and $i,j=1,\dots, l$, we have:
    \begin{equation*} \begin{split}&\textstyle (\Lambda_t \Lambda_t^T)_{ij}=\sum _{k=1}^l \Lambda _{ik}\Lambda _{kj}^T = \sum _{k=1}^l \Lambda _{ik}\Lambda _{jk}=\sum _{k=1}^l m_i (\Gamma_1)_{ik} m_j (\Gamma_1)_{jk}=\\
    &\textstyle = m_i m_j \sum _{k=1}^l (\Gamma_1)_{ik}(\Gamma_1)_{jk}=  m_i m_j (\Gamma_1 \Gamma _1^T)_{ij}
    \end{split}\end{equation*}
    and similarly for $\bar{\Lambda}$ we get: $(\bar{\Lambda}_t\bar{\Lambda}_t^T)_{ij}^{-1}=\frac{1}{m_im_j}(\Gamma_2 \Gamma_2^T)_{ij}^{-1}$. 
    Thus, we have showed that the integrand in \eqref{equaz14} is equal to { $\frac{1}{2}\Big(tr\big((\Gamma_2 \Gamma_2^T)^{-1}(\Gamma_1 \Gamma_1^T)\big)-l - \log \frac{\det (\Gamma_1 \Gamma_1^T)}{\det (\Gamma_2 \Gamma_2^T)}\Big)$},  exactly as in the specific relative entropy computed in Lemma \ref{lem3.5}. In this sense we can say here, as in Remark \ref{rem:eq_first_ex}, that Gantert's inequality is an equality.
\end{rmk}

\bibliographystyle{abbrv}
\bibliography{joint_biblio(5)}

\begin{thebibliography}{10}

\bibitem{Al22b}
D.~Aldous.
\newblock {\em notes}, {\tt https://www.stat.berkeley.edu/\textasciitilde aldous/Research/OP/ent-MG.pdf}.

\bibitem{Al22a}
D.~Aldous.
\newblock What is the max-entropy win-probability martingale?
\newblock {\tt https://www.stat.berkeley.edu/\textasciitilde aldous/Research/OP/maxentmg.html}.

\bibitem{AvFrHoSa97}
M.~Avellaneda, C.~Friedman, R.~Holmes, and D.~Samperi.
\newblock Calibrating volatility surfaces via relative-entropy minimization.
\newblock {\em Applied Mathematical Finance}, 4(1):37--64, 1997.

\bibitem{BeBe23}
J.~Backhoff-Veraguas and M.~Beiglb{\"o}ck.
\newblock The most exciting game.
\newblock {\em Electronic Communications in Probability}, 29:1--12, 2024.

\bibitem{BaPa22}
J.~Backhoff-Veraguas and G.~Pammer.
\newblock Applications of weak transport theory.
\newblock {\em Bernoulli}, 28(1):370--394, 2022.

\bibitem{BaUn22}
J.~Backhoff-Veraguas and C.~Unterberger.
\newblock {On the specific relative entropy between martingale diffusions on the line}.
\newblock {\em Electronic Communications in Probability}, 28(none):1 -- 12, 2023.

\bibitem{BaWaZh24}
J.~Backhoff-Veraguas, Z.~Wang, and X.~Zhang.
\newblock The {A}ldous martingale in arbitrary dimensions.
\newblock {\em Work-in-progress}, 2024.

\bibitem{BaZh24}
J.~Backhoff-Veraguas and X.~Zhang.
\newblock Specific {W}asserstein divergence between continuous martingales.
\newblock {\em arXiv preprint arXiv:2404.19672}, 2024.

\bibitem{BeChLo24}
J.-D. Benamou, G.~Chazareix, and G.~Loeper.
\newblock From entropic transport to martingale transport, and applications to model calibration.
\newblock {\em arXiv preprint arXiv:2406.11537}, 2024.

\bibitem{CaGo07}
P.~Cattiaux and N.~Gozlan.
\newblock Deviations bounds and conditional principles for thin sets.
\newblock {\em Stochastic processes and their applications}, 117(2):221--250, 2007.

\bibitem{ChCoReWa24}
F.~Chen, G.~Conforti, Z.~Ren, and X.~Wang.
\newblock Convergence of {S}inkhorn's algorithm for entropic martingale optimal transport problem.
\newblock {\em arXiv preprint arXiv:2407.14186}, 2024.

\bibitem{CoDo22}
A.~Cohen and Y.~Dolinsky.
\newblock A scaling limit for utility indifference prices in the discretised {B}achelier model.
\newblock {\em Finance Stoch.}, 26(2):335--358, 2022.

\bibitem{DM18}
H.~De~March.
\newblock Entropic approximation for multi-dimensional martingale optimal transport.
\newblock {\em arXiv preprint arXiv:1812.11104}, 2018.

\bibitem{DMHL19}
H.~De~March and P.~Henry-Labordere.
\newblock Building arbitrage-free implied volatility: {S}inkhorn's algorithm and variants.
\newblock {\em arXiv preprint arXiv:1902.04456}, 2019.

\bibitem{Fo22b}
H.~F\"{o}llmer.
\newblock Doob decomposition, {D}irichlet processes, and entropies on {W}iener space.
\newblock In {\em Dirichlet forms and related topics}, volume 394 of {\em Springer Proc. Math. Stat.}, pages 119--141. Springer, Singapore, [2022] \copyright 2022.

\bibitem{Fo22a}
H.~F\"{o}llmer.
\newblock Optimal couplings on {W}iener space and an extension of {T}alagrand's transport inequality.
\newblock In {\em Stochastic analysis, filtering, and stochastic optimization}, pages 147--175. Springer, Cham, [2022] \copyright 2022.

\bibitem{Ga91}
N.~Gantert.
\newblock {\em Einige grosse {A}bweichungen der {B}rownschen {B}ewegung}, volume 224 of {\em Bonner Mathematische Schriften [Bonn Mathematical Publications]}.
\newblock Universit\"{a}t Bonn, Mathematisches Institut, Bonn, 1991.
\newblock Dissertation, Rheinische Friedrich-Wilhelms-Universit\"{a}t Bonn, Bonn, 1991.

\bibitem{GuPoRe23}
G.~Guo, D.~Possama{\"\i}, and C.~Reisinger.
\newblock Randomness and early termination: what makes a game exciting?
\newblock {\em ArXiv e-prints}, 2306.07133, 2023.

\bibitem{He19}
P.~Henry-Labordere.
\newblock From (martingale) {S}chr{\"o}dinger bridges to a new class of stochastic volatility model.
\newblock {\em Available at SSRN 3353270}, 2019.

\bibitem{Ka83}
R.~L. Karandikar.
\newblock On the quadratic variation process of a continuous martingale.
\newblock {\em Illinois journal of Mathematics}, 27(2):178--181, 1983.

\bibitem{Be24}
E.~Kimani~Bellotto.
\newblock Specific relative entropy between continuous martingales.
\newblock {\em University of Vienna (Masterarbeit)}, 2024.

\bibitem{NuWi24}
M.~Nutz and J.~Wiesel.
\newblock On the martingale {S}chr{\"o}dinger bridge between two distributions.
\newblock {\em arXiv preprint arXiv:2401.05209}, 2024.

\end{thebibliography}

\appendix
\section{ Elementary Properties of the Entropy}
\label{app:A}

We collect here a few elementary properties of the relative entropy. We start by recalling the well-knwon expression for the relative entropy between Gaussians:

\begin{ex}\label{ex1.1}
Given two $l$-multivariate normal distributions $ \mathbb{P}_{i} \sim \mathcal{N}_l(\mu_{i}, \Sigma_{i})$, where $i=1,2$, $l \in \mathbb{N}$, $\mu _i \in \mathbb{R}^l, \Sigma_i \in \mathbb{R}^{l \times l}$, then the relative entropy between the two variables is:
\begin{equation} \label{H(N|N)}\textstyle 
    H\big(\mathbb{P}_1|\mathbb{P}_2)=\frac{1}{2}\big(tr(\Sigma_2^{-1}\Sigma_1)-l+(\mu_2-\mu_1)^T\Sigma_2^{-1}(\mu_2-\mu_1)-\log( \frac{\det \Sigma_1}{\det \Sigma_2})\big).
\end{equation}
\end{ex}

\begin{rmk}\label{rmk2}
    If the  two normal variables have the same mean $\mu \in \R^l$ then:
    \begin{equation*}\textstyle 
        H(\mathcal{N}_l(\mu, \Sigma_1)|\mathcal{N}_l(\mu, \Sigma_2))=\frac{1}{2}\Big(tr(\Sigma_2^{-1}\Sigma_1)-l-\log( \frac{\det \Sigma_1}{\det \Sigma_2})\Big).
    \end{equation*}
    \end{rmk}

The next well-known lemmas are also useful throughout our work:

    \begin{lem} \label{lem1.2}
     Let $(X,\mathcal{A})$ and $(Y,\mathcal{B})$ be measurable spaces, $\mu$ and $\nu$ probability measures on $(X,\mathcal{A})$ such that $\mu \ll \nu$ and $T:X \rightarrow Y$ such that $\mathcal{A}=\sigma(T):=\sigma(\{T^{-1}(B):B \in \mathcal{B}\})$. Then $T(\mu)\ll T(\nu)$ and for $\nu$-a.e. $x \in X$:
    \begin{equation*}\textstyle 
        \frac{dT(\mu)}{dT(\nu)}(T(x))=\frac{d\mu}{dv}(x) \mbox{ and }  H(\mu|\nu)=H(T(\mu)|T(\nu)).
    \end{equation*}
\end{lem} 
 
\begin{lem} \label{lem1.3}
    For a random vector $(V,W)$ we denote by $\mu_{V,W}$ its law and by $\mu_W^{V=v}$ a
regular conditional distribution of $W$ given $V=v$. Let $(X_1,\dots,X_n)$ and $(Y_1,\dots,Y_n)$ be random vectors. Then:
\begin{equation*} \begin{split}
&H(\mu_{X_1, \dots, X_n}|\mu _{Y_1,\dots,Y_n})= H(\mu_{X_1}|\mu_{Y_1})+\\ &\textstyle +\sum _{k=2}^n \int H\big( \mu _{X_k}^{X_1=x_1,\dots,X_{k-1}=x_{k-1}}|\mu _{Y_k}^{Y_1=x_1,\dots,Y_{k-1}=x_{k-1}}\big)d\mu_{X_1, \dots, X_{k-1}}(x_1,\dots,x_{k-1}),
\end{split} \end{equation*}
which   we can also write as: 
\begin{equation*} \begin{split} 
&\textstyle H\big(\mathcal{L}((X_1, \dots,X_n))|\mathcal{L}((Y_1, \dots,Y_n))\big)=\\ &\textstyle H(\mathcal{L}(X_1)|\mathcal{L}(Y_1))
+\sum _{k=2}^n \int H\Big(\mathcal{L}(X_k|X_1=x_1, \dots , X_{k-1}=x_{k-1})\\&\textstyle | \,\mathcal{L}(Y_k|Y_1=x_1, \dots , Y_{k-1}=x_{k-1}) \Big) d\mathcal{L}(X_1, \dots,X_{k-1} )dx_1 \dots dx_{
k-1}.
\end{split}\end{equation*} 
\end{lem}

\begin{lem}
     Let $\mu, \nu$ be probability measures on $(X,\mathcal{A})$ s.t.\ $H(\mu|\nu)<\infty$. Let $\mathcal{B}\subset \mathcal{A}$ be a sub-sigma algebra and $\mu|_{\mathcal{B}},\nu|_{\mathcal{B}}$ the restrictions of $\mu$ and $\nu$ to $\mathcal{B}$.Then\\
    (i) \[\textstyle \frac{d\mu|_{\mathcal{B}}}{d\nu|_{\mathcal{B}}}=\mathbb{E}_{\nu}\Big[(\frac{d\mu}{d\nu})|\mathcal{B}\Big].\]
    (ii) \[H(\mu|_{\mathcal{B}}|\nu |_{\mathcal{B}}) \leq H(\mu |\nu ).\]
\end{lem}

\section{Pending Proofs and Missing Arguments}
\label{app:B}

\begin{proof}[Proof of Remark \ref{ex_extra}] {Let $B_t^l \sim \mathcal{N}_l(\mathbf{1}, t \mathcal{I})$ be the $l$-dimensional Brownian motion at time $t$ starting from $B_0=\mathbf{1}$, and let $b_t$ its density function. Let $M_t^{\Gamma}$, defined as $dM_t^{\Gamma}:=diag(M_t^{\Gamma})\Gamma dB_t$, be a $\R ^l$-valued martingale Black-Sholes model with parameter $\Gamma \in \R^{l \times l}$ such that $M_0^{\Gamma}=\mathbf{1}$ and $m_t$ indicates its density function. In particular, $M_t^{\Gamma}$ is distributed as a lognormal variable with mean $\mu $ and variance $\Sigma$ where $\mu_i= -\frac{t}{2} \sum_{k=1}^l \Gamma_{ik}^2$ and $\Sigma_{ij}=t(\Gamma \Gamma^T)_{ij}$ for each  $i,j=1, \dots, l$.  \\
We look now at the logarithms of the respective densities $\forall x \in \R^l$: 
\begin{equation*} \begin{split} 
\log(m_t(x))&=\textstyle  \log \Big( \frac{1}{\sqrt{(2 \pi)^l \det \Sigma}} \prod _{i=1}^l \frac{1}{x_i} \exp \big( -\frac{1}{2}(\log x-\mu )^T \Sigma ^{-1}(\log x-\mu )\big)\Big)
\\&=\textstyle 
-\frac{1}{2} \log (2\pi)^l \det \Sigma - \sum _{i=1}^l \log x_i -\frac{1}{2}\big((\log x-\mu)^T\Sigma^{-1} (\log x - \mu )\big).
\end{split}
\end{equation*}
The last term, as we already discussed so far, becomes:
\begin{equation*} \begin{split} 
&\textstyle -\frac{1}{2} tr\big( \Sigma^{-1}((\log x - \mu )(\log x - \mu )^T)\big)\\&\textstyle =-\frac{1}{2} tr\big( \Sigma^{-1}( \log x \log x ^T - 2\log x \mu^T + \mu \mu^T)).
\end{split}
\end{equation*}
While on the other hand, 
   \begin{equation*} \begin{split} 
\log(b(x))&\textstyle =\log \Big( \frac{1}{\sqrt{(2\pi)^l \det (t \mathcal{I})}} \exp \big(-\frac{1}{2}(x-\mathbf{1})^T (t \mathcal{I})^{-1}(x-\mathbf{1})\big)\Big)\\&\textstyle =-\frac{1}{2}\log (2\pi )^l \det (t \mathcal{I}) -\frac{1}{2}(x-\mathbf{1})^T (t \mathcal{I})^{-1}(x-\mathbf{1})\end{split}
\end{equation*} 
and, by some linear algebra already discussed, the last term becomes: \begin{equation*} \begin{split} 
\textstyle -\frac{1}{2} tr \big( (t \mathcal{I})^{-1}((x-\mathbf{1}) (x-\mathbf{1})^T ) \big)=-\frac{1}{2} tr \big( \frac{1}{t} \mathcal{I}(xx^T -2 x \mathbf{1}^T + \mathbf{1}\mathbf{1}^T)\big).
\end{split}
\end{equation*}
Before computing their expectations we look at the expectation of $M_t^{\Gamma}$: 
\begin{equation*} \begin{split} 
\big(\E[M_t^{\Gamma}]\big)_i= e ^{\mu_i+ \frac{1}{2}\Sigma_{ii}}&=e^{-\frac{t}{2} \sum_{k=1}^l \Gamma_{ik}^2+ \frac{1}{2} t(\Gamma \Gamma^T)_{ii}}\\&=e^{-\frac{t}{2} \sum_{k=1}^l \Gamma_{ik}^2+ \frac{t}{2} \sum_{k=1}^l \Gamma_{ik}^2}=1
\end{split}
\end{equation*}
as $i=1, \dots l$, and the variance is:  
\begin{equation*} \begin{split} 
\big(Var[M_t^{\Gamma}]\big)_{ij}&= \textstyle \exp \Big(\mu_i+\mu_j+ \frac{1}{2}(\Sigma_{ii}+\Sigma_{jj})\Big)(e^{\Sigma_{ij}}-1)\\&= \textstyle \exp \Big(-\frac{t}{2} \sum_{k=1}^l \Gamma_{ik}^2 -\frac{t}{2} \sum_{k=1}^l \Gamma_{jk}^2 + \frac{1}{2}(t(\Gamma \Gamma ^T)_{ii} +t(\Gamma \Gamma ^T)_{jj})\Big)(e^{\Sigma_{ij}}-1)\\&\textstyle =\exp \Big(-\frac{t}{2} \sum_{k=1}^l \Gamma_{ik}^2 -\frac{t}{2} \sum_{k=1}^l \Gamma_{jk}^2 + \frac{t}{2}(\sum_{k=1}^l \Gamma_{ik}^2 +\sum_{k=1}^l \Gamma_{jk}^2)\Big)(e^{\Sigma_{ij}}-1)\\&\textstyle = e^{\Sigma_{ij}}-1=e^{t\sum_{k=1}^l \Gamma_{ik} \Gamma_{jk}}-1
\end{split}
\end{equation*}
as $i,j=1, \dots l $, while the second moment is: 
\begin{equation*} \begin{split} 
\big(\E_{M_t^{\Gamma}}[M_t^{\Gamma} (M_t^{\Gamma})^T]\big)_{ij}&=(\E[M_t^{\Gamma}]\E[M_t^{\Gamma}]^T)_{ij}+(Var[M_t^{\Gamma}])_{ij}\\&= (\mathbf{1}\mathbf{1}^T)_{ij}+(e^{\Sigma_{ij}}-1)=e^{t\sum_{k=1}^l \Gamma_{ik}\Gamma_{jk}}.
\end{split}
\end{equation*}
Moreover, we recall that $\log M_t^{\Gamma} \sim \mathcal{N}_l (\mu, \Sigma)$. Now we can compute the expectations of the densities. 
\begin{equation*} \begin{split} 
\E_{M_t^{\Gamma}}\big[\log m(x)\big]&\textstyle = -\frac{1}{2} \log (2\pi)^l \det \Sigma - \sum_{i=1}^l \mu_i - \frac{1}{2} tr (\Sigma^{-1}(\Sigma+\mu \mu^T -2 \mu \mu ^T +\mu\mu^T))\\&\textstyle = -\frac{1}{2} \log (2\pi)^l \det \Sigma -  \sum_{i=1}^l \mu_i - \frac{1}{2}  tr (\mathcal{I})\\&\textstyle = -\frac{1}{2} \log (2\pi)^l \det \Sigma -  \sum_{i=1}^l \mu_i - \frac{1}{2} l
\end{split}
\end{equation*}
which, recalling that $\mu_i= -\frac{t}{2} \sum_{k=1}^l \Gamma_{ij}^2$ and $\Sigma_{ij}=t(\Gamma \Gamma^T)_{ij}$  $ \forall i,j=1, \dots, l$, becomes:
\begin{equation*} \begin{split}
&=\textstyle -\frac{l}{2}\log 2 \pi t -  \frac{1}{2} \log (\det (\Gamma \Gamma^T) ) - \sum_{i=1}^l \mu_i -\frac{1}{2}l 
\\&\textstyle =-\frac{l}{2}\log 2 \pi t-  \frac{1}{2} \log (\det (\Gamma \Gamma^T) ) +\frac{t}{2} \sum_{i=1}^l \sum_{k=1}^l \Gamma_{ik}^2  - \frac{1}{2} l.
\end{split}
\end{equation*}
On the other hand,
\begin{equation*} \begin{split} 
\E_{M_t^{\Gamma}}\big[ \log b(x) \big]&\textstyle =-\frac{l}{2} \log 2 \pi t - \frac{1}{2}tr(\frac{1}{t } \mathcal{I}(e^{\Sigma}  -\mathbf{1}\mathbf{1}^T))\\&\textstyle =
-\frac{l}{2} \log 2 \pi t - \frac{1}{2}tr(\frac{1}{t } \mathcal{I}e^{\Sigma} )+\frac{1}{2}(\mathbf{1}^T\frac{1}{t} \mathcal{I}\mathbf{1})\\&\textstyle =-\frac{l}{2} \log 2 \pi t - \frac{1}{2t} \sum_{i=1}^{l}e^{t\sum
_{k=1}^l\Gamma_{ik}^2}+ \frac{l}{2t}.
\end{split}
\end{equation*}
Finally, we get the desired relative entropy:
\begin{equation*} \begin{split} 
&\textstyle H(\mathcal{L}(M_t^{\Gamma})| \mathcal{L}(B_t^l))= \E_{M_t^{\Gamma}}\big[\log m(x)\big] -E_{M_t^{\Gamma}}\big[ \log b(x) \big] \\&\textstyle = -  \frac{1}{2} \log (\det (\Gamma \Gamma^T) ) +\frac{t}{2}  \sum_{i=1}^l \sum_{k=1}^l \Gamma_{ik}^2  - \frac{1}{2} l 
+ \frac{1}{2t} \sum_{i=1}^{l}e^{t\sum
_{k=1}^l\Gamma_{ik}^2}- \frac{l}{2t} \\&\textstyle =\frac{1}{2}\Big(-\log (\det (\Gamma \Gamma^T) ) + t \sum_{i=1}^l \sum_{k=1}^l \Gamma_{ik}^2 -l + \sum_{i=1}^l \frac{e^{t\sum
_{k=1}^l\Gamma_{ik}^2}-1 }{t} \Big).
\end{split}
\end{equation*}}
\end{proof}

\begin{proof}[Proof of Lemma \ref{F_l} ]
Clearly
\begin{equation*} \begin{split}
            F_l(\alpha \mathcal{I})&\textstyle =\frac{1}{2}\big(tr(\alpha \mathcal{I})-l-\log ( \det( \alpha \mathcal{I}))\big)       =\frac{1}{2}(l\alpha -l-l \log \alpha)\\
            &=\frac{l}{2}(\alpha -1-l \log \alpha)
            =lF_1(\alpha).
\end{split}\end{equation*}
For the second point:
\begin{equation*} \begin{split}
    F_l(D)&\textstyle =\frac{1}{2}\big(tr(D)-l-\log ( \det( D))\big)
    =\frac{1}{2}(\sum_{i=1}^l \sigma _i ^2-\sum_{i=1}^l 1- \sum _{i=1}^l \log \sigma _i^2)\\
    &\textstyle =\sum _{i=1}^l \frac{1}{2}(\sigma_i^2 - 1 -\log \sigma_i^2)
    = \sum _{i=1}^l F_1(\sigma _i^2 ).
\end{split}\end{equation*}
For the third point:
\begin{equation*} \begin{split}
    F_l(\Sigma)&\textstyle =\frac{1}{2}\big(tr(\Sigma)-l-\log(det \Sigma)\big)
    =\frac{1}{2}(tr(D)-l-\log(det D))\\
    &\textstyle =F_l(D)=\sum_{\sigma^2\in eigenvalues\, of \, \Sigma} F_1(\sigma^2).
\end{split}\end{equation*}  
For the fourth point, first we settle convexity: We need to prove that for $\lambda \in [0,1]$ and $X,Y \in \mathbb{R}^{l \times l}$ symmetric positive definite, we have $\lambda F_l(X)+(1-\lambda)F_l(Y)\geq F_l(\lambda X +(1-\lambda)Y)$. Clearly this boils down to showing:
    \[\log(\det (\lambda X +(1-\lambda)Y)) \geq \lambda \log (\det X) +(1-\lambda)\log (\det Y).\] This is well-known, but we give an argument here for the convenience of the reader.
    Passing to the exponentials, this becomes: \[ \det (\lambda X + (1-\lambda )Y)  \geq (\det X)^{\lambda}(\det Y)^{1-\lambda}.\]
    Since, without loss of generality, $X$ is positive definite such that $X^{1/2}X^{1/2}=X$ the left hand side can be written as: \[ \det (X^{1/2}(\lambda \mathcal{I} + (1-\lambda) X^{-1/2} YX^{-1/2})X^{1/2})=\det X \det (\lambda \mathcal{I}+(1-\lambda ) Z)\] for  $Z:= X^{-1/2}YX^{-1/2}$. While the right hand side:
    \[\det X^{1/2} (\det X^{-1/2}YX^{-1/2} )^{1-\lambda}\det X^{1/2}= \det X^{1/2} (\det Z)^{1-\lambda}\det X^{1/2} .\]
    So we need to show that \[ \det(\lambda \mathcal{I}+(1-\lambda)Z)\geq (\det Z)^{1-\lambda}.\]
    Further, we notice that $Z$ is clearly symmetric and also positive definite. Indeed, the smallest eigenvalue of $Z$ is given by \begin{equation*}
        \begin{split}
           \textstyle  \min _{|\mu|=1}\mu ^T(X^{-1/2}YX^{-1/2}) \mu &\textstyle = \min _{|\mu|=1}(X^{-1/2} \mu)^TY(X^{-1/2} \mu)\\&\textstyle =\min _{|\mu|=1} |X^{-1/2} \mu|^2\frac{(X^{-1/2} \mu)^T}{|X^{-1/2} \mu|}Y \frac{(X^{-1/2} \mu)}{|X^{-1/2} \mu|}.
            \end{split}
    \end{equation*}
 Now considering 
    the eigenvalues of $Z$, $(\gamma_i)_{i=1, \dots, l}$,  it remains to show: 
    \[\textstyle \prod_{i=1}^l(\lambda +(1-\lambda)\gamma_i) \geq \prod_{i=1}^l \gamma_i^{1-\lambda},\] 
    i.e.\
    \[\textstyle \sum_{i=1}^l \log (\lambda+(1-\lambda)\gamma_i) \geq (1-\lambda)\sum_{i=1}^l \log \gamma_i \]
    which holds by the concavity of the logarithm term by term. 

    Now we observe that $F_l(X)\geq 0$ and $F_l(X)= 0$ iff $X=\mathcal I$. Indeed, if $(\sigma_i)_i$ are the eigenvalues of $X$, then we know that $F_l(X)=\sum_i F_1(\sigma_i^2)$. The result follows since clearly $F_1(\sigma^2)\geq 0$ and $F_1(\sigma^2) =0$ iff $\sigma^2=1$.
\end{proof}

\end{document}